\documentclass[11pt]{amsart}

\usepackage{amsmath,amssymb,amscd,amsthm,stmaryrd}

\usepackage{amssymb}
\usepackage{graphicx}
\usepackage{amsfonts}
\usepackage{bbm}
\usepackage{mathrsfs}
\usepackage[colorlinks=true, pdfstartview=FitV, linkcolor=blue, citecolor=blue, urlcolor=blue]{hyperref}
\usepackage{color}
\usepackage{graphicx,epstopdf,verbatim}
\usepackage[cal=boondoxo]{mathalfa}

\usepackage{tikz-qtree}

\usepackage{stackengine}
\stackMath
\newcommand\tsup[2][2]{%
 \def\useanchorwidth{T}%
  \ifnum#1>1%
    \stackon[-.5pt]{\tsup[\numexpr#1-1\relax]{#2}}{\scriptscriptstyle\sim}%
  \else%
    \stackon[.5pt]{#2}{\scriptscriptstyle\sim}%
  \fi%
}

\allowdisplaybreaks

\newcounter{dummy} \numberwithin{dummy}{section}

\newtheorem{definition}[dummy]{Definition}
\newtheorem{lemma}[dummy]{Lemma}
\newtheorem{theorem}[dummy]{Theorem}
\newtheorem{corollary}[dummy]{Corollary}
\newtheorem{proposition}[dummy]{Proposition}
\newtheorem{example}[dummy]{Example}
\newtheorem{remark}[dummy]{Remark}

\newtheorem{ansatz}[dummy]{Ansatz}
\newtheorem{innercustomthm}{Ansatz}
\newenvironment{mythm}[1]{\renewcommand\theinnercustomthm{#1}\innercustomthm} {\endinnercustomthm}

\DeclareMathOperator{\sgn}{sgn}

\providecommand{\keywords}[1]
{
  \small	
  \textbf{\textit{Keywords---}} #1
}

\newcommand{\R}{\mathbb{R}}

\newcommand{\T}{\mathcal{T}}

\begin{document}
\tikzset{every tree node/.style={minimum width=2em,draw},
         blank/.style={draw=none},
         edge from parent/.style=
         {draw,edge from parent path={(\tikzparentnode) -- (\tikzchildnode)}},
         level distance=1.5cm}

\title{Further evidence towards the Fourier Entropy-Influence conjecture}
\author{Mar\'ia Jos\'e Gonz\'alez }
\thanks{M.J.G. was supported in part by the Spanish Ministerio de Ciencia e Innovaci\'on (grant no. PID2021-123151NB-I00), and by the
grant ``Operator Theory: an interdisciplinary approach", reference ProyExcel$_-$00780, a project financed in the 2021 call for Grants for
Excellence Projects, under a competitive bidding regime, aimed at entities qualified as Agents of the Andalusian Knowledge System,
in the scope of the Andalusian Research, Development and Innovation Plan (PAIDI 2020). Counseling of University, Research and
Innovation of the Junta de Andaluc\'ia.}
\address{Departamento de Matemáticas, Universidad de C\'adiz}
\email{majose.gonzalez@uca.es}
\author{Paul MacManus}
\address{Trevally Capital}
\email{paul.macmanus@gmail.com}
\author{Mar\'ia Cristina Pereyra}
\address{Department of Mathematics and Statistics, University of New Mexico}
\email{crisp@math.unm.edu}


\begin{abstract} The Fourier Entropy-Influence (FEI) conjecture states that the Fourier entropy of Boolean functions is uniformly bounded by their total influence. It has been verified for canonical examples such as disjoint tribes and for some classes of Boolean functions such as symmetric functions and read-$k$ decision trees (with a constant that depends linearly on $k$). In this note we present new classes of Boolean functions that  verify the FEI conjecture. The key element  is an inequality controlling  the difference between the entropy of a function $f$ and the average of the entropies of $f^{\pm}$, the sub-functions obtained by setting $x_m=\pm1$ for some $m$, by the  $m$-influence of $f$. If this key inequality were to hold for Boolean functions, then the full FEI conjecture would follow by induction. We introduce the notion of a stopping binary tree and observe that functions that satisfy the key inequality  at the branching nodes of the tree and the FEI conjecture at the stopping nodes will satisfy the FEI conjecture. We identify some classes of Boolean functions that  fit this framework: the \emph{$\delta$-tribes functions}, the monotone Boolean  functions with the \emph{tribe separation property}, and the Boolean functions with the \emph{semi-separation property}, and, along the way, demonstrate some results that we hope the experts in this fascinating field might find useful.

\end{abstract}


\maketitle

\tableofcontents


\section{Introduction}\label{sec:intro}

The $n$-dimensional Hamming  or Boolean cube is  $\{-1,1\}^{n}$. Let  $\mathcal{B}_n$ denote the  Boolean functions $f$ defined on the $n$-dimensional Hamming cube,  i.e. $f:\{-1,1\}^n\to \{-1,1\}$.

Let $ [n]:=\{1,2,\dots ,n\}$. Given $S\subset [n]$, the   \emph{parity function associated to}  $S$ is given by   $\chi_S(x):=\prod_{i\in S} x_i $ for all $S\neq \emptyset$ and $x=(x_1,\dots,x_n)\in\{-1,1\}^n$, and $\chi_{\emptyset}(x)=1$.  The collection of parity functions, $\{\chi_S: S\subset [n]\}$, forms an orthonormal  basis for the set of real-valued functions defined on the Hamming cube, with inner product $\langle f,g\rangle =\frac{1}{2^n}\sum_{x\in\{-1,1\}^n} f(x)g(x)$. Therefore, each such $f$ has a Fourier expansion
 $f= \sum_{S\in [n]} \widehat{f}(S)\, \chi_S$, where $\widehat{f}(S)=\langle f,\chi_S\rangle$.

 The \emph{total influence} and  \emph{spectral entropy} of $f:\{-1,1\}^n\to \R$, denoted respectively by ${\bf I}(f)$ and ${\bf H}(f)$,  are given in terms of the Fourier coefficients of $f$,  by 
 \[ {\bf I}(f)=\sum_{S\in [n]} |S| |\widehat{f}(S)|^2,  \quad\quad\quad{\bf H}(f) = \sum_{S\subset [n]} |\widehat{f}(S)|^2\log \frac{1}{|\widehat{f}(S)|^2},\]
 where $\log t = \log_2 t$ for $t>0$, and $ t\log \frac{1}{t}$ is defined to be 0 when $t=0$.

The famous Fourier Entropy-Influence  (FEI) conjecture of Friedgut and Kalai~\cite{FK96} states that for Boolean functions the entropy can be uniformly bounded by the total influence.\\

 \noindent{\bf FEI Conjecture.}
  \emph{ There is  $K>0$, independent of the dimension,  such that  for all Boolean functions $f$ we have that   
   $${\bf H}(f)\leq K \, {\rm {\bf I}}(f).$$}

As described in  \cite{ACK+20}:  ``Intuitively, the Fourier entropy measures how `spread out' the Fourier distribution is over the $2^n$ subsets of $[n]$ and the total influence measures the concentration of the Fourier distribution on the `high' level coefficients. Informally, the FEI conjecture states that Boolean functions whose Fourier distribution is well `spread out'  (i.e., functions with large Fourier entropy) must have significant Fourier weight on the high-degree monomials (i.e., their total influence is large)". For a quantification of this statement see \cite[Commentary surrounding (1.2)]{Han25}.

A well known first approach to the FEI conjecture provides a logarithmic bound. More precisely,  if $f\in\mathcal{B}_n$ then ${\bf H}(f)\leq c \log n\, {\bf I}(f)$.  If the FEI conjecture were proven to hold for Boolean functions, we could deduce from it the celebrated KKL (Kahn-Kalai-Linial) Theorem~\cite{KKL88}. The KKL Theorem  states that there is a variable whose influence is larger than $n/\log n$ times the variance of any  function $f\in\mathcal{B}_n$. The KKL Theorem allows one to show the  related weaker Fourier min-Entropy-Influence (FMEI) conjecture holds for monotone Boolean functions, which is still unknown for Boolean functions.  See~\cite{KKL+20} and the survey \cite[Sections 4-5]{GMcP25}.

The  FEI conjecture is  false for real-valued functions on the Hamming cube (as opposed to Boolean functions).  It suffices to concentrate the Fourier coefficients on  singletons (small sets), say $\widehat{f}(\{i\})=1/\sqrt{n}$ for all $i\in [n]$ and all other Fourier coefficients are zero. Then ${\bf H}(f)=\log n$ and $ {\bf I}(f)=1$. However, one cannot concentrate all Fourier coefficients on singletons for Boolean functions (see~\cite[Exercise 1.5]{O'Do21}). 

 The FEI conjecture holds for  ANDs, ORs, majority  and uniform disjoint tribes functions, via straightforward calculations. It also holds for random disjunction of conjunctions of Boolean (DNF's) \cite{KLW10}, symmetric functions and read-once decision trees \cite{OWZ11}, and read-once formulas \cite{OT13, CKL+13} using more subtle arguments.  There are results that fall short of the conjecture, as the constant $K$ depends on some additional parameter: for example, Boolean functions that can be computed by a read-$k$ decision tree ($K=9k$), and  Boolean functions that can be computed by a decision tree whose expected depth is $d$ ($K=12d$), see \cite{WWW14}.

 The conjecture has not been proven for monotone functions, a natural and important class which can be represented by  general tribes. However,  it is known for uniform disjoint tribes, where direct computations can be done \cite{BL90}, and for non-uniform disjoint tribes using a beautiful composition theorem of O'Donnell and Tan~\cite{OT13}.   In this article, we present  proper subfamilies of  Boolean functions, the \emph{$\delta$-tribes functions},   the monotone Boolean functions with the  \emph{tribes separation property},  and  the Boolean functions with the \emph{semi-separation property},  for which the FEI conjecture holds. These families include the non-uniform disjoint tribes, as well as some general tribes functions and  non-monotone functions outside the classes for which positive results are known. This adds evidence towards the validity of the conjecture. 

To obtain our results,  we draw inspiration from a well-known induction argument approach to the FEI conjecture. Let $\mathcal{A}$ denote a generic class of Boolean functions and  $\mathcal{A}_n=\mathcal{A}\cap \mathcal{B}_n$. Our goal  is to show  there is a $K>0$ such that ${\bf H}(f) \leq K\,{\bf I}(f)$ for all $f \in \mathcal{A}_n$ and for all $n\geq 1$, by induction on the dimension $n$.  As the base case $n=1$ holds, we are left to argue that if the inductive statement holds for dimension $n-1$ then it holds for dimension $n$.

For $f\in\mathcal{B}_n$, denote ${\bf I}_m(f)$ its influence for the $m^{th}$-variable, defined in terms of Fourier coefficients by ${\bf I}_m(f)= \sum_{S\in [n]: S\ni m} |\widehat{f}(S)|^2.$  It is well-known that 
$${\bf I}(f) = \frac{{\bf I}(f^+) + {\bf I}(f^-)}{2} + {\bf I}_m(f),$$
where $f^{\pm}\in \mathcal{B}_{n-1}$ correspond to splitting $f$ by the variable $x_m$, i.e. 
$$f^{\pm}(x_1,\dots, x_{m-1},x_{m+1},\dots x_{n})=f(x_1,\dots, x_{m-1},\pm1, x_{m+1},\dots, x_n).$$
 The inductive step will hold if we can show that there is $m\in [n]$ such that
\begin{equation}\label{eq:basicInduction}
 {\bf H}(f)-\frac{{\bf H}(f^+) + {\bf H}(f^-)}{2} \leq K\, {\bf I}_m(f).
 \end{equation}
 This leads us to our first ansatz. 
 \begin{ansatz}\label{ansatz1} 
Let $\mathcal{A}$ be a class of Boolean functions. Suppose  there is a constant $K>0$ (independent of the dimension) such that every $f\in\mathcal{A}$ can be split into functions $f^{\pm}\in\mathcal{A}$ (using a variable $m$ that can depend on $f$) for which  inequality~\eqref{eq:basicInduction} holds. Then the FEI conjecture holds in $\mathcal{A}$, i.e. ${\bf H}(f)\leq K\,{\bf I}(f)$ for all $f\in \mathcal{A}$.
\end{ansatz}

It is well known that ${\bf I}_m(f) = \frac14\sum_{S\in [n]\setminus\{m\}} \big (\widehat{f^+}(S) -\widehat{f^-}(S)\big )^2$. We will show that \eqref{eq:basicInduction} is equivalent to 
\begin{equation}\label{eq:basicStep}
\sum_{S\in [n]\setminus\{m\}} \big (\widehat{f^+}(S)^2 +\widehat{f^-}(S)^2\big )\, \psi\left [\frac{ \frac12\big (\widehat{f^+}(S) -\widehat{f^-}(S)\big )^2}{\widehat{f^+}(S)^2 +\widehat{f^-}(S)^2} \right ] \leq K \sum_{S\in [n]\setminus\{m\}} \big (\widehat{f^+}(S) -\widehat{f^-}(S)\big )^2,
\end{equation}
where $\psi(t)=t\log \frac{1}{t}+(1-t)\log\frac{1}{1-t}$ for $t\in (0,1)$ and $\psi(0)=\psi(1)=0$ is the ``binary entropy".

If we could show~\eqref{eq:basicStep}  for all Boolean functions, or for all monotone Boolean functions, we could deduce the FEI conjecture for either class by Ansatz~\ref{ansatz1}. The key here being that both Boolean functions  and monotone Boolean functions are closed under splitting by any  variable.

Very recently, Xiao Han has shown that if one could prove an  inequality similar to~\eqref{eq:basicStep}  for all Boolean functions, then  the FEI conjecture would follow by induction, see  \cite[Remark 3.1, and (3.17)] {Han25}.  The arguments that lead Han to his inequality  are different to the ones presented in this paper.

One can  show that \eqref{eq:basicStep} holds  if 
\begin{equation}\label{eq:basicStronger}
\sum_{S\in [n-1]} \big |\widehat{f^+}(S)^2 -\widehat{f^-}(S)^2\big | \leq K \sum_{S\in [n-1]} \big (\widehat{f^+}(S) -\widehat{f^-}(S)\big )^2.
\end{equation}
This inequality  had already been  identified by Han \cite[Question 3.1]{Han25}.

As it turns out, we can construct examples of  monotone Boolean functions in arbitrarily large  dimension  where splitting by the last variable  will prevent~\eqref{eq:basicStronger} 
from holding for any fixed $K>0$. However, the same functions when split by a well-chosen variable, will satisfy~\eqref{eq:basicStronger}  for a fixed $K$ independent of the dimension, hence they will satisfy~\eqref{eq:basicStep}. These monotone functions $F$ we call $\delta$-tribes, and they have the property that when we split with the wrong variable, $F^+=f$ is a disjoint tribes function and $F^-=f-2\delta_y$ where $y$ is a boundary point of $f$, that is $f(y)=1$ and $f(x)=-1$ for all $x<y$. However, when split with an appropriate choice of variable, $F^+$ is another $\delta$-tribes function and $F^-$ is a disjoint tribes function. This leads us to state our second  ansatz (for the definition of a stopping binary tree, see Section~\ref{sec:binaryDecisionTree}).
\begin{ansatz}\label{ansatz2} 
Let $\mathcal{A}$ be a class of Boolean function such that for all  $f\in \mathcal{A}$  there is a stopping binary decision  tree such that~\eqref{eq:basicInduction}  holds with constant $K>0$ for the splitting variable at the branching nodes, and at stopping nodes the  FEI conjecture holds with constant $K$, then ${\bf H}(f)\leq K\, {\bf I}(f).$
\end{ansatz}
We can replace  inequality~\eqref{eq:basicInduction} by~\eqref{eq:basicStep}  or~\eqref{eq:basicStronger} in both Ansatz~\ref{ansatz1} and~\ref{ansatz2}. 
The rest of the paper consists of  constructing classes of Boolean functions to which we can apply  Ansatz~\ref{ansatz1} or~\ref{ansatz2}. 
In addition to the $\delta$-tribes, we can show that inequality~\eqref{eq:basicStronger} holds  for certain Boolean functions or monotone Boolean functions that have some degree of variable separation. First, when a Boolean function $f$ is \emph{semi-separated}, meaning there is a variable $x_m$ such that when splitting $f$ by $x_m$, the splitting functions $f^{\pm}$ are semi-separated, namely $f^+=pq$ and $f^-=pr$ where the Boolean functions $p$, $q$ and $r$ depend on disjoint sets of variables. Second, when a monotone Boolean function $f$ is \emph{tribe separated by a variable}, meaning  there is a variable $x_m$ such that the tribes containing $x_m$ are disjoint from the tribes not containing $x_m$. We say a Boolean function has the \emph{separation property} or, if monotone, the  \emph{tribes separation property}, if it has a stopping binary decision tree such that the functions at the branching nodes are semi-separated, or, if monotone, they are tribe separated by the splitting variable at the node, and in both cases, at the stopping nodes the functions satisfy the  FEI conjecture.

\begin{theorem} The $\delta$-tribes functions, the monotone Boolean functions with the tribe separation property, and the Boolean functions with the separation property satisfy the FEI conjecture.
\end{theorem}
To the best of our knowledge the classes of functions we describe do not fall into any of the previously known categories of  functions that satisfy the FEI conjecture. In that sense we provide further evidence for the validity of the conjecture. 

In Section~\ref{sec:prelim}, we present basic definitions and useful facts about Boolean functions, in particular the Fourier decomposition and canonical examples, including  monotone functions. We also introduce the concepts of Fourier entropy, individual and total influence, as well as binary decision tree representations and stopping binary trees.  We also explore the relation between a binary tree representation of a Boolean function and its total influence. 

In Section~\ref{sec:monotone}, we present some useful results about the structure of monotone functions as tribes, including the well known Sperner's Lemma, and some constructions of monotone functions with properties that will be useful later.  In Appendix~A (Section~\ref{sec:appendix}) we further explore the structure of monotone functions. 

In Section~\ref{sec:InductiveApproachFEI}, we describe the inductive approach to the FEI conjecture, that led us to state  Ansatz~\ref{ansatz1} and ~\ref{ansatz2}. The calculation that proves that~\eqref{eq:basicStep} is equivalent to \eqref{eq:basicInduction}  is in Appendix~B (Section~\ref{sec:algAcrobatics}).

In Section~\ref{sec:(counter)examples}, we show that the wrong choice of splitting variable leads to the failure of the stronger inequality~\eqref{eq:basicStronger}, however when choosing the right variable for the same example, inequality ~\eqref{eq:basicStronger} holds with $K=38$. This leads us to introduce the class of \emph{$\delta$-tribes} that will satisfy the assumptions of  Ansatz~\ref{ansatz2} with disjoint tribes at the stopping nodes. 

In Section~\ref{sec:NewClassesBooleanFunctions}, we introduce the new classes of Boolean functions with  the \emph{semi-separation property}, and  of monotone Boolean functions that satisfy the \emph{tribe separation property}. In both cases, the appropriate separation  property will ensure the assumptions in Ansatz~\ref{ansatz2} are met. 
 We describe specific examples that show these are  not empty classes.

\section{Preliminaries}\label{sec:prelim}

In this section we present basic definitions and useful facts about Boolean functions (Section~\ref{sec:Boolean}), in particular the Fourier decomposition (Section~\ref{sec:Fourier}), and canonical examples, including  monotone functions (Section~\ref{sec:canonical-examples}). We also introduce the concepts of Fourier entropy, individual and total influence (Section~\ref{sec:EntropyInfluence}), as well as binary decision tree representations and stopping binary trees (Section~\ref{sec:binaryTree}). Finally we highlight the relation between binary stopping trees and total influence (Section~\ref{sec:InfluenceTrees}).

\subsection{Hamming cube, Boolean functions, and characters}\label{sec:Boolean}

The $n$-dimensional Hamming  or Boolean cube is   $\{-1,1\}^{n}$. Our main players, the Boolean functions are defined on a Hamming cube and output $\pm1$.
Let  $\mathcal{B}_n$ denote the  Boolean functions $f$ defined on the $n$-dimensional Hamming cube,  i.e. $f:\{-1,1\}^n\to \{-1,1\}$. 

Canonical examples include: parity functions, the dictator function, majority function, max and min functions,\footnote{In the Boolean function literature the min is the OR functions and the max is the AND function.}   and tribes. Most of these are \emph{monotone functions}, an interesting and important class  of Boolean functions,  see Section~\ref{sec:canonical-examples}. The book \emph{``Analysis of Boolean functions"} by Ryan O'Donnell~\cite{O'Do21}, is the standard reference and is a fantastic resource for anyone interested in this area.

Given $S\subset [n]:=\{1,2,\dots ,n\}$ the  \emph{character} or \emph{parity function associated to}  $S$ is given by
 \[\chi_{\emptyset}(x_1,\dots,x_n)=1, \quad\quad \chi_S(x_1,\dots,x_n):=\prod_{i\in S} x_i \;\;\;\mbox{for all}\; S\neq \emptyset.\] 
Note that the character functions are Boolean functions.  

 The delta functions $\delta_a$,  are defined  for $a\in \{-1,1\}^n$ as $\delta_a(x)=1$ if $x=a$ and 0 otherwise.  The collection $\{\delta_a: a\in \{-1,1\}^n\}$ forms a basis for the $2^n$-dimensional vector space $\mathcal{R}_n$ of real-valued functions defined on the $n^{th}$-Hamming cube.  
 It is the analogue to the standard basis in Euclidean space. 
These are not Boolean functions, however it will be useful to keep in mind their expansion in terms of characters.
For $a=(a_1,\dots, a_n)\in \{-1,1\}^n$ one can verify that for all $x=(x_1,\dots,x_n)\in \{-1,1\}^n$
\begin{eqnarray}
\delta_a(x) & = & \left (\frac{1+a_1x_1}{2}\right )\left (\frac{1+a_2x_2}{2}\right )\dots\left (\frac{1+a_nx_n}{2}\right ) 
\, = \, \frac{1}{2^n}\sum_{S\subset [n]} \left ( \prod_{i\in S}a_i \right ) \chi_S(x). \label{eq:expansion-standard-basis-in-parity-functions}
\end{eqnarray}
It follows that the  collection $\{\chi_S: S\subset [n]\}$ forms a  basis  for $\mathcal{R}^n$. Therefore every function $f:\{-1,1\}^n\to \R$ has an expansion in terms of characters.

\subsection{Fourier-Walsh decomposition}\label{sec:Fourier}

The vector space $\mathcal{R}^n$  equipped with the    inner product  defined by 
\[ \langle f,g\rangle := {\bf E}[fg]= \frac{1}{2^n}\, \sum_{x\in \{-1,1\}^n}f(x)g(x),\]
becomes a finite dimensional Hilbert space, denoted $L^2(\{-1,1\}^n)$.
We are using the expectation ${\bf E}$ with uniform probability $1/2^n$ for each element $x\in \{-1,1\}^n$. 
The norm induced by this inner product is given by $\|f\|_2=\sqrt{\langle f,f\rangle}$. Note that all Boolean functions have  $L^2$-norm one. One can verify that the characters are orthonormal with respect to this inner product, hence they form an orthonormal basis in $L^2(\{-1,1\}^n)$.

The Boolean functions $\mathcal{B}_n$ are a subset of $L^2(\{-1,1\}^n)$ (in fact a subset of the $L^2$-unit ball)  but not a subspace. We can still take advantage of the Fourier decomposition in terms of the characters, that is for all $f\in L^2(\{-1,1\}^n)$ (whether $f$ is Boolean or not),
\[ f(x) = \sum_{S\subset [n]} \widehat{f}(S) \chi_S(x), \quad\quad\mbox{where}\;\; \widehat{f}(S)=\langle f,\chi_S\rangle.\]

The Fourier decomposition is unique, that is if $f(x)=\sum_{S\subset[n]} a_S\chi_S$ then $a_S=\widehat{f}(S)$. In particular,  we can read from \eqref{eq:expansion-standard-basis-in-parity-functions} the Fourier coefficients of  $\delta_a$, 
\begin{equation}\label{eq:Fourier-coeff-1_a}
 \widehat{\delta_a}(S)=  \frac{1}{2^n}\prod_{i\in S}a_i =  \frac{1}{2^n}\,\chi_S(a).
 \end{equation}
Parseval's identity holds, that is for all $f\in L^2(\{-1,1\}^n)$ we have that
\[  \|f\|_2^2 = \frac{1}{2^n}\sum_{x\in \{-1,1\}^n} |f(x)|^2= \sum_{S\subset [n]} |\widehat{f}(S)|^2. \] 
We can also write the expectation 
 in terms of the Fourier coefficients for functions $f\in \mathcal{R}_n$:
\[ {\bf E}[f]=\widehat{f}(\emptyset), \quad \quad {\bf E}[f\chi_S]=\widehat{f}(S).\]
The following independence property will be useful.

\begin{remark}\label{remark2}
Given functions $f,g\in L^2(\{-1,1\}^n)$ such that $f$ depends only on a set of variables indexed by $S_f\subset [n]$, and $g$ depends only on a set of variables indexed by $S_g\subset [n]$. If $S_f\cap S_g=\emptyset$ then
\[ \widehat{fg}(S) = \left \{\begin{array}{cc} \widehat{f}(S\cap S_f) \, \widehat{g}(S\cap S_g) & \mbox{if $S \subset S_f\cup S_g$}\\
                                                           0 & \mbox{if $S\setminus(S_f\cup S_g) \neq \emptyset $}\end{array}\right. 
                                                           \]
\end{remark}

\subsection{Canonical examples}\label{sec:canonical-examples}
In this section we present some examples that stem from the mathematics of \emph{social choice}, the point here is how to aggregate opinions of many agents. A Boolean function $f\in\mathcal{B}_n$ can be used as a model for an election among two candidates, the inputs are the votes of the $n$ electors, and the output of the function is the winner.

\begin{example}[Dictator function]
The dictator functions are the characters corresponding to the singletons $\chi_{\{i\}}$. 
 The name should be self-explanatory, the election is decided by the $i^{th}$ voter, ``the dictator".
 For $S\subset [n]$, $\widehat{\chi}_{\{i\}} (S)=1$ for  $S= \{i\}$ and $\widehat{\chi}_{\{i\}} (S)=0$ otherwise.
\end{example}

Among the most familiar voting methods is the majority rules. The majority function outputs what the majority of voters vote for. 
\begin{example}[Majority function]\label{eg:majority}
For $n$ odd, the majority function ${\rm Maj}_n: \{-1,1\}^n\to \{-1,1\}$ is defined by ${\rm Maj}_n(x)=\sgn (x_1+x_2+\dots +x_n)$. 
\end{example}

The Boolean AND and OR functions correspond to voting rules in which a candidate is elected unless all voters are unanimously opposed. 
Denote  $e=(1,1,\dots, 1) \in \{-1,1\}^n$, where the dimension is to be understood by the context.

\begin{example}[AND or $\max$ function]\label{eg:AND}
The function ${\rm AND}_n:\{-1,1\}^n\to\{-1,1\}$ is defined by 
\[{\rm AND}_n (x)= \left \{ \begin{array}{cc} +1 & {x} \neq -e,\\
                                                                             -1 & { x}  = - e.
                                                                          \end{array}\right. \]
Note that ${\rm AND}_n(x)=\max\{x_1,x_2,\dots, x_n\}$.   We will  use   the notation $\max$ instead of ${\rm AND}$. 

Observe that
\begin{equation}\label{eq:AND}
\max\{x_1,x_2,\dots, x_n\}= 1-2\delta_{-e}(x)=1-\frac{(1-x_1)\dots (1-x_n)}{2^{n-1}},
\end{equation}
hence $\widehat{{\rm AND}}_n(\emptyset)=1-1/2^{n-1}$, $\widehat{{\rm AND}}_n(S)=(-1)^{|S|+1}/2^{n-1}$ for all $S\neq \emptyset$.   

\end{example}

In particular, we will repeatedly use  formula~\eqref{eq:AND}  when $n=2$,
\begin{equation}\label{eq:max}
\max\{x_1,x_2\} = 1-\frac12(1-x_1)(1-x_2) =\frac12 \big ( 1+x_1+x_2-x_1x_2\big ).
\end{equation}

       \begin{example}[OR or $\min$  function]\label{eg:OR}
The function ${\rm OR}_n:\{-1,1\}^n\to\{-1,1\}$ is defined by 
\[ {\rm OR}_n ( x)= \left \{ \begin{array}{cc} -1 & {x} \neq e,\\
                                                                            +1 & { x}  = e.
                                                                             \end{array}\right. \]
Note that ${\rm OR}_n(x)=\min\{x_1,x_2,\dots, x_n\}$.  
 We  will use  the notation $\min$ instead of ${\rm OR}$.
 
Observe that
\begin{equation}\label{eq:OR}
\min\{x_1,x_2,\dots,x_n\}= -1+2\delta_e(x) = -1+\frac{(1+x_1)\dots (1+x_n)}{2^{n-1}},
\end{equation}                                                                             
hence $\widehat{{\rm OR}}_n(\emptyset)=-1+1/2^{n-1}$, $\widehat{{\rm OR}}_n(S)=1/2^{n-1}$ for all $S\neq \emptyset$.   
 \end{example}       
                                  
\begin{example}[General Tribes]  A  \emph{general  tribes} function $f\in\mathcal{B}_n$ is defined by
\[ 
F=\max\{g_1,g_2,\dots g_s\},
\]
where  $ g_j (x)= \min\{x_i: i\in \Lambda_j\}$ for a collection of  subsets $\Lambda_j$ of $ [n]$. We will often call the sets $T_j=\{x_i: i\in \Lambda_j\}$ the tribes, and write $g_j=\min T_j$. 
\end{example}

\ A general tribes function where the $\Lambda_j$ are pairwise disjoint, is called a \emph{disjoint tribes}  function. We can use the polynomial expansions for max and min given by~\eqref{eq:AND} and~\eqref{eq:OR} to read the Fourier coefficient of  disjoint tribes functions. In particular, when $S=\emptyset$, we have that:
\begin{equation}\label{eq:DisjointTribesEmptysetCoeff}
\widehat{F}(\emptyset) = 1- \frac{1}{2^{s-1}}\prod_{j=1}^s \left (2-\frac{1}{2^{w_j-1}}\right )=1- 2\prod_{j=1}^s \left (1-\frac{1}{2^{w_j}}\right ) , 
\end{equation}
where $w_j$ is the number of elements in $\Lambda_j$. When $n=sw$, and $n_j=w$ for all $j\in [s]$, we say  $F$ is  a \emph{uniform disjoint tribes function}, denoted ${\rm Tribes}_{s,w} $ ($s$ is the ``size'' or number of tribes, and $w$ is the ``width''). Explicit formulas for all the Fourier coefficients of ${\rm Tribes}_{s,w} $ can be found in  \cite[Proposition 4.14]{O'Do21}.

Finally, we recall the class of monotone Boolean functions.

\begin{definition}\label{def:monotone}
A Boolean function $f\in\mathcal{B}_n$ is  \emph{monotone}  if $f({x}) \leq f({ y})$ whenever $x\leq y$,  defined to mean $x_i\leq y_i$ for all $i\in [n]$. Denote by $\mathcal{M}_n$ the class of monotone functions $f\in\mathcal{B}_n$.
\end{definition}
It is easy to see that the dictator, majority, max, min and general tribes functions are all monotone. In Section~\ref{sec:monotone} we will show that monotone functions coincide with the class of general tribes.

\subsection{Entropy and Influence}\label{sec:EntropyInfluence}

We define the spectral entropy for functions in $L^2(\{-1,1\}^n)$. We first define the total influence, and the individual influence  for $f\in \mathcal{B}_n$ , and deduce a Fourier representation for these quantities that can be used to define total and individual influence for $f\in L^2(\{\{-1,1\}^n)$.
Here $\log x = \log_2 x$.

Given $f\in L^2((\{-1,1\}^n)$ its \emph{spectral entropy}\footnote{This coincides with  the Shannon entropy of the probability  distribution $|\widehat{f}(S)|^2$ defined on $S\subset [n]$.}  ${\bf H}(f)$ is defined to be
\[ {\bf H}(f) = \sum_{S\subset [n]} |\widehat{f}(S)|^2\log \frac{1}{|\widehat{f}(S)|^2}.\]

  Given $f\in\mathcal{B}_n$, the \emph{$i^{th}$-influence of $f$}, denoted  ${\bf I}_i(f)$,  is defined to be the probability that when we change the value of the $i^{th}$-coordinate, we get different outcomes, namely,
  \[ {\rm {\bf I}}_i(f)= {\rm Prob}\left \{f(x^{(i\to 1)})\neq f(x^{(i\to -1)})\right \},\]
  where  we are using the notation $x^{(i\to b)}=(x_1,\dots,x_{i-1},b,x_{i+1},\dots, x_n)$.

   The \emph{total influence} of $f\in \mathcal{B}_n$ is
  $ {\rm {\bf I}}(f)= \sum_{i\in [n] } {\rm {\bf I}}_i(f).$

 The \emph{$i^{th}$ (discrete) derivative of $f\in \mathcal{R}_n$} for $i\in [n]$, is  defined by
\[D_if(x) = \frac{f(x^{(i\to1)})-f(x^{(i\to-1)})}{2}.\]
 Notice that $D_if(x)^2=0$ if $f(x^{(i\to1)})=f(x^{(i\to-1)})$ (in other words, the $i^{th}$ coordinate {\sc is not} pivotal for $x$), and $D_if(x)^2=1$  if $f(x^{(i\to1)})\neq f(x^{(i\to-1)})$ (in other words, the $i^{th}$ coordinate {\sc is} pivotal for $x$). Therefore 
 the individual influence ${\rm {\bf I}}_i(f)$ is the square of the $L^2$-norm of $D_if$,
\begin{equation}\label{eq:derivative-influence}
 {\bf I}_i(f)= {\bf E} [ (D_if)^2]= \|D_if\|_2^2.
 \end{equation}

The discrete derivative operator acts like formal differentiation on parity functions: $D_i\chi_S = \chi_{S\setminus\{i\}}$, and hence on Fourier expansions by linearity. More precisely, if $f:\{-1,1\}^n\to \R$ and $f(x)=\sum_{S\subset [n]}\widehat{f}(S) \chi_S(x)$, then
\begin{equation}\label{eq:derivativeFourier}
 D_if(x)=\sum_{S\subset [n]:S\ni i}\widehat{f}(S) \chi_{S\setminus \{i\}}(x),
 \end{equation}
and we can compute its $L^2$-norm using Parseval. We  conclude that 
\begin{equation}\label{eq:indInfluenceForuier}
{\rm {\bf I}}_i(f)= \sum_{S\subset [n]:S\ni i}|\widehat{f}(S)|^2,  
\end{equation}
 and  therefore,   the total influence  of $f$   can be given in terms of Fourier coefficients as follows:
\begin{equation}\label{eq:InfluenceFourier}
 {\rm {\bf I}}(f)=\sum_{S\subset [n]} |S| |\widehat{f}(S)|^2.
 \end{equation}
 We can now use~\eqref{eq:derivative-influence} or~\eqref{eq:indInfluenceForuier} to define individual influence, and~\eqref{eq:InfluenceFourier}  to define total influence for any function  $f:\{-1,1\}^n\to \R$, not just Boolean functions.
When $\|f\|_2=1$, we can think of $\{ |\widehat{f}(S)|^2: S\subset [n]\}$ as a probability measure in $\mathcal{P}([n])=\{S: S\subset [n]\}$, then ${\rm {\bf I}}(f)$, the total influence of $f$, measures the expected size of a subset $S$ of $[n]$.

In the case of monotone Boolean functions, the relation between the total influence and the Fourier coefficients is greatly simplified, see~\cite[Proposition 2.21]{O'Do21}.

\begin{proposition}\label{prop:influenceMonotone}
Let  $f\in \mathcal{M}_n$  then  ${\bf I}_i(f)=\widehat{f}(\{i\})$ and  ${\bf I}(f) = \sum_{i\in [n]} \widehat{f}(\{i\})$.

\end{proposition}

\begin{proof}
Since $x^{(i\to -1)}\leq x^{(i\to1)}$ and $f$ is monotone, then $D_if(x)$ takes values in $\{0,1\}$, hence $|D_if|^2=D_if$. Thus 
$
{\bf I}_i(f) = {\bf E} [ D_if]= \widehat{D_if}(\emptyset)=\widehat{f}(\{i\}).
$
The last inequality holds by \eqref{eq:derivativeFourier}, since $S\setminus\{i\}=\emptyset$ precisely when $S=\{i\}$.
Consequently, by definition ${\bf I}(f) = \sum_{i\in [n]} \widehat{f}(\{i\})$.
\end{proof}

\subsection{Binary trees and stopping  trees}\label{sec:binaryTree}

The notion of a binary decision tree representation for Boolean functions is very useful.

\subsubsection{Splitting by a variable}

Any $f\in \mathcal{B}_n$ can be split in two halves, 
$f^{\pm}\in\mathcal{B}_{n-1}$, by a variable $x_m$ with $m\in [n]$, by fixing $x_m=1$ or $x_m=-1$, as follows:  
\[f^{\pm}(x^m) = f(x^{(m\to\pm 1)}), \quad\mbox{where} \; x^m=(x_1,\dots,  x_{m-1}, x_{m+1},\dots, x_n)\in \{-1,1\}^{n-1}.\] 
We can recover $f$ from the splitting functions $f^{\pm}$ and the splitting variable $x_m$ as follows:
\begin{equation}\label{eq:fsplitxm}
f(x) = \frac{f^+(x^m) + f^-(x^m)}{2} + x_m\frac{f^+(x^m) - f^-(x^m)}{2}.
\end{equation}
  We will write $f = \frac{f^++f^-}{2} + x_m\frac{f^+-f^-}{2}$, where it is understood that the splitting has been done by the variable $x_m$.  In particular, 
 the functions $f^{\pm}$ do not depend on $x_m$,  therefore
 $D_mf =\frac{f^+ - f^-}{2}$,   and by \eqref{eq:derivative-influence} and Parseval we conclude that
  \begin{equation}\label{eq:influence-fpm}
   {\bf I}_m(f)= \frac14{\bf E}[|f^+-f^-|^2] = \frac14\sum_{S\subset [n]\setminus\{m\}} \big (\widehat{f^+}(S)-\widehat{f^-}(S)\big )^2
  \end{equation}
  We also deduce from~\eqref{eq:fsplitxm} that if $S\subset [n]$ and $m\notin S$ then
 \begin{equation}\label{eq:fhat-fpmhat}
  \widehat{f}(S) =\frac{\widehat{f^+}(S) + \widehat{f^-}(S)}{2}, \quad \mbox{and} \quad  \widehat{f}(S\cup \{m\}) =\frac{\widehat{f^+}(S) - \widehat{f^-}(S)}{2}.
  \end{equation}
 
 When the splitting variable is not specified, it will be understood to be  the last variable, that is $m=n$ when $f\in\mathcal{B}_n$. 
  If we iterate the splitting, we create a decision tree associated to the Boolean function $f\in \mathcal{B}_n$. 
  
  \subsubsection{Binary decision trees}\label{sec:binaryDecisionTree}
  
 A \emph{binary decision tree associated to $f\in \mathcal{B}_n$}  is a rooted binary tree  that has $f$ as the root at level 0 and the internal nodes  at level $k\in [n-1]$ are labeled by $\rho\in \sigma_k$, where $\sigma_k$ denotes  the set of sequences of length $k$ of  $\pm$ signs. We will agree that $\sigma_0=\emptyset$. If $\rho\in \sigma_k$ then $\rho\pm \in \sigma_{k+1}$ by concatenating  $\pm$ to the right of the  sequence $\rho$. For example, when $k=2$ then $\sigma_2=\{ ++, +-, -+, --\}$, and if $\rho=+-$ then $\rho+=+-+$. Note that $\rho$ encodes the path from the root to the node. At node $\rho$ we have a branching function $f^\rho$ and a
  label  $i_\rho\in [n]$ (corresponding to  the splitting variable at the node). At level 0, we have just one node $\rho=\emptyset$, the branching function is the root $f^{\emptyset}=f$ and the label  of the splitting variable is $i_{\emptyset}\in [n]$. Each variable $i\in [n]$ appears at most once along any path from the root  to a node.  At node $\rho$,  the left branch  corresponds to assigning $x_{i_\rho}=-1$ and the right branch to assigning $x_{i_\rho}=+1$, that is  $f^{\rho\pm}=(f^\rho)^{\pm}$. 
  For example,  the functions $f^{\pm}$ dictated by the chosen splitting variable at node $\rho=\emptyset$  will appear in  the first branching level, and we will continue branching for  at most $n$ levels, as once a splitting variable has been chosen it disappears from any path that goes through that node.
 When a branch is a constant  function, $\pm 1$, we have reached a leaf, no further splitting is necessary as the value along that path is now determined. If this happens before level $n$,  say at level $k$, this indicates that the Boolean function is constant on a subcube of dimension $n-k$.
 
 We will  consider \emph{stopping binary decision trees} where we stop the branching before reaching the leaves. At the stopping nodes  we record the stopping functions $f^\rho$,  at the branching nodes we record the branching functions $f^\rho$ and the splitting variables $x_{i_\rho}$. When depicting binary trees, it often suffices to record at level 0 the root and the splitting variable label $i_{\emptyset}$, at the stopping nodes the functions $f^\rho$ and at the branching nodes the splitting variable label $i_\rho$, see Figure~\ref{fig:trees}.

 Given the root and the splitting variables at every node, we can find the function $f^\rho$  at  each node $\rho$,  by definition, starting from the root and proceeding down systematically.  Using repeatedly~\eqref{eq:fsplitxm}, we can find the root given the functions at the stopping nodes and the splitting variables at the branching nodes. For example, for the  tree  on the right in Figure~\ref{fig:trees},
 \[ f=\frac12\Big [f^{-}+ \Big (\frac12(f^{+-}+f^{++}) +\frac{x_4}{2}(f^{++}-f^{+-})\Big )\Big ]  + \frac{x_1}{2}\Big [ \Big (\frac12(f^{+-}+f^{++}) +\frac{x_4}{2}(f^{++}-f^{+-})\Big ) - f^{-}\Big ] .\] 

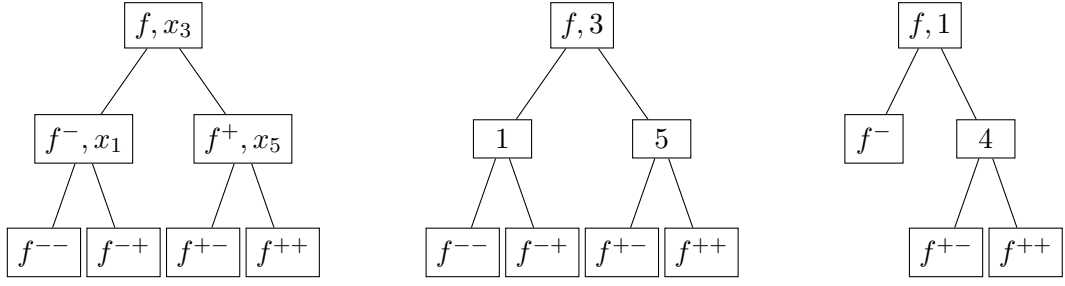
\begin{figure}\label{fig1}
 
\begin{tikzpicture}
\Tree
[.$f,x_3$
    [.$f^-,x_1$
    \edge[];{$f^{--}$}
     \edge[];{$f^{-+}$}
     ]
    [.$f^+,x_5$
    \edge[];{$f^{+-}$}
     \edge[];{$f^{++}$}
     ]
  ]
\end{tikzpicture}
\hskip .5in 
\begin{tikzpicture}
\Tree
[.$f,3$
    [.$1$
    \edge[];{$f^{--}$}
     \edge[];{$f^{-+}$}
     ]
    [.$5$
    \edge[];{$f^{+-}$}
     \edge[];{$f^{++}$}
     ]
  ]
\end{tikzpicture}
\hskip .5in 
\begin{tikzpicture}
\Tree
[.$f,1$
    \edge[];{$f^-$}
    [.$4$
    \edge[];{$f^{+-}$}
     \edge[];{$f^{++}$}
     ]
  ]
\end{tikzpicture}

\caption{\footnotesize{The left and middle figures depict  the same  stopping binary tree. 
 The right figure  depicts  a stopping  binary tree for the same root function with stopping nodes at different levels, note that
  the stopping functions at level 2  are  different  from those  in the first two trees. } }\label{fig:trees}
 \end{figure}

 We can associate to $f\in\mathcal{B}_n$ an enormous number of binary trees and identifying  those trees with a simpler structure is a fruitful endeavor.  The following definitions of depth-$d$ decision trees and read-once trees  come from~\cite{OWZ11}, and read-$k$ decision trees from~\cite{WWW14}.
  
We say that $f\in \mathcal{B}_n$ is \emph{computable as a depth-$0$ decision tree} if it is constantly $-1$ or $1$. We
inductively say that $f$ is computable as a \emph{depth-$d$ decision tree} if there is a coordinate $i\in [n]$ such
that the functions $f^{\pm}$ obtained when splitting $f$ by the variable $x_i$ are  computable by depth-$(d-1)$ decision trees.
 We further say that the decision-tree computation is \emph{read-once} if $f^{\pm}$ depend on disjoint sets of coordinates and are themselves
inductively read-once. In our language it means each splitting variable is queried no more than once  in the entire tree. \emph{Read-$k$ decision trees} is 
 the class of decision trees in which each variable is queried at no more than $k$ distinct locations
in the entire tree.

 \subsection{Influence and binary trees} \label{sec:InfluenceTrees} 
 
 Given a Boolean function $f$ and a binary tree with root $f$, we can write ${\bf I}(f)$, the influence of  $f$ as a weighted sum of  ${\bf I}_{i_\rho}(f^\rho)$, the individual influences of the splitting functions on a binary tree at each node.
 The key is the  following lemma, see~\cite[Proposition 2.4]{WWW14} and some of its consequences.

  \begin{lemma}\label{lem:I_k}
Let $f\in\mathcal{B}_n$ and $m\in [n]$.  If we split $f$ by the variable $x_m$, then for all $k\in [n]\setminus\{m\}$, 
$$ {\bf I}_k(f) = \frac12 \big ({\bf I}_k(f^+) +{\bf I}_k(f^-)\big ).$$ 
\end{lemma}
 
Using the definition of influence and Lemma~\ref{lem:I_k} we  get 
\begin{equation}\label{lem:induction-influence2}
{\bf I}(f)=\frac12\big [{\bf I}(f^+) + {\bf I}(f^-)\big ] + {\bf I}_m(f).
\end{equation}

We can iteratively use~\eqref{lem:induction-influence2} at each node of  any binary tree  associated to $f$, say with functions $f^\rho$ and nodes $i_\rho$  to obtain
\begin{equation}\label{eq:non-uniformInfluence}
{\bf I}(f) =  \sum_{k=0}^{n-1} \frac{1}{2^k} \sum_{\rho \in\sigma_k} {\bf I}_{i_\rho} (f^\rho),
\end{equation}
where we understand $\sigma_0=\emptyset$ and $i_{\emptyset}$ to be the node at root $f^{\emptyset}=f$. 

We could have stopped branching at different levels to obtain a \emph{stopping tree} associated to $f$, and there will be an appropriate sum of multiples of the total influence of the  functions at the stopping nodes, and a sum accounting for all the branching nodes  involving individual influences of the branching  functions  corresponding to the splitting variable at those nodes.
Let us introduce some notation to be more precise.

Given a Boolean function $f$ and a stopping tree $\T$. We have two flavors of nodes, the \emph{stopping nodes} denoted $\sigma_{\T}$ and \emph{branching nodes}, denoted $\beta_{\T}$. Given a node $\rho$ in $\T$, it has associated a level $k_\rho$ (the number of $\pm$ signs in the sequence),  and the branching nodes also have a splitting label $i_\rho$ (indicating  we are splitting by the variable $x_{i_\rho}$). 
Then the formula we will get for the influence of the Boolean function $f$ by repeatedly using~\eqref{lem:induction-influence2}, is
\begin{equation}\label{eq:stoppingTreeInfluenceFourier}
{\bf I}(f) = \sum_{\rho \in \sigma_{\T}} \frac{1}{2^{k_\rho}}{\bf I}(f^\rho) +  \sum_{\rho\in\beta_{\T}}  \frac{1}{2^{k_\rho}} {\bf I}_{i_\rho}(f^\rho).\end{equation}

\section{Monotone Boolean functions}\label{sec:monotone}
 
 We collect some lemmas about monotone Boolean functions that will be used in constructing our classes of functions that satisfy the FEI conjecture. This includes the well known Sperner's Lemma (Section~\ref{sec:SpernersLemma}) identifying monotone Boolean functions and general tribes,  as well as other results including (i) two characterizations of monotone Boolean  functions $f$ in terms of  $f^{\pm}$, (ii) the existence of  monotone Boolean functions with Fourier transform with arbitrarily  large $L^1$-norm,  and  (iii) a mechanism to generate non-constant monotone Boolean functions $f$  given $f^+$ where $f^-$ is a minimal perturbation of $f^+$  (Section~\ref{sec:structureMonotone1}). We believe most of these results must be  folklore in the field. 
 In Appendix A (Section~\ref{sec:appendix}) we state and prove some further structural results for monotone Boolean  functions.

 \subsection{Sperner's Lemma}\label{sec:SpernersLemma}
 In this section we will show that  non-constant, monotone Boolean functions are exactly the general tribes.   We will write $\mathcal{M}^*_n$ for the class of non-constant monotone functions in $\mathcal{B}_n$.

\begin{definition}[Sperner Family] A \emph{Sperner family} for $[n]$ is a finite collection of non-empty subsets,  $\mathcal{S}= \{\Lambda \subset [n]\}$, such that no $\Lambda\in \mathcal{S}$ is a proper subset of  any other $\widetilde{\Lambda}\in \mathcal{S}$.
\end{definition}

For a general  tribes function,  if a label set contains another label set, then the larger label set  is redundant, simply because if $\Lambda \subset \widetilde{\Lambda} \subset [n]$ then $ \min\{x_j: j\in \widetilde{\Lambda}\} \leq \min\{x_j: j\in \Lambda\}$, and therefore
\[
\max\{\min\{x_j: j\in \widetilde{\Lambda}\},\min\{x_j: j\in \Lambda\}\}= \min\{x_j: j\in \Lambda\}.
\]
It follows that every general tribes function has a minimal representation, one where the label sets form a Sperner family. Sperner's Lemma will provide such a minimal  tribes representation for non-constant monotone Boolean functions.

\begin{lemma}[Sperner's Lemma]\label{lem:Sperner's} 
Any  $f\in \mathcal{M}^*_n$ can be represented as a minimal tribes function.
\end{lemma}
\begin{proof}
Identify each set $\Lambda\subset [n]$ with an element of $x_{\Lambda}\in \{-1,1\}^n$ a string of $\pm 1$, as follows: if $i\in \Lambda$ then $x_i=1$ and if $i\notin \Lambda$ then $x_i=-1$. This is a one-to-one correspondence. For example, given $\Lambda=\{2,3,5\}\subset [5]$ then $x_{\Lambda}=(-1,1,1,-1,1)$, and given $x = (1,-1,1,-1,-1)$ then $\Lambda_x=\{1,3\}\subset [5]$. Note that  $\Lambda \subset \widetilde{\Lambda}$ if and only if  $x_{\Lambda}\leq x_{\widetilde{\Lambda}}$.

Note that since $f$ is non-constant and monotone, it must be that $f(e)=f(x_{[n]}) =1$ and $f(-e) = f(x_{\emptyset})=-1$.
Choose the collection of minimal sets $\Lambda$ for which $f(x_{\Lambda}) =1$, meaning that if $\Lambda'\subsetneq \Lambda$ then $f(x_{\Lambda'})=-1$. Such minimal sets exist because the function is monotone, and it is a non-empty collection because $f(x_{\emptyset})=-1$. These  minimal sets give the Sperner family $\mathcal{S}$ we are looking for. 
It remains to show that the monotone function $f$ coincides with the general tribes function determined by the Sperner family $\mathcal{S}$, namely $g=\max\{g_{\Lambda}: \Lambda\in{\mathcal S}\}$.

We want to show that $f(y)=g(y)$ for all $y\in \{-1,1\}^n$. It suffices to show that $f(y)=1$ if and only if $g(y)=1$. First, suppose $f(y)=1$. Then the set $\Lambda_y=\{i\in [n]: y_i=1\}$ must contain a minimal set $\Lambda\in \mathcal{S}$ and its corresponding point $x_{\Lambda}$ by definition has the property that $f(x_{\Lambda})=1$. Now the point $x_{\Lambda}$ has $i^{th}$-coordinate equal to 1 precisely when $i\in \Lambda$, and by construction $\Lambda\subset \Lambda_y$, therefore $x_{\Lambda}\leq y$ which means that $y_i=1$ for all $i\in \Lambda$. This implies that  $g_{\Lambda}(y)=1$ and thus $g(y)=1$. Second, suppose now that $g(y)=1$. Then there is $\Lambda\in \mathcal{S}$ such that $g_{\Lambda}(y)=1$, which implies that $y_i=1$ for all $i\in \Lambda$, and therefore $x_{\Lambda}\leq y$. But by construction $f(x_{\Lambda})=1$ for all $\Lambda\in \mathcal{S}$, and by monotonicity   $f(y)=1$.
\end{proof}

For example, consider $f\in \mathcal{M}^*_3$,  given by 

\begin{center}
\begin{tabular}{c|r|c|c}\hline
$x=(x_1,x_2,x_3)$ & $f(x)$ & $\Lambda_x$ & in $\mathcal{S}$ \\ \hline
$(+1,+1,+1)$ & 1 & \{1,2,3\} & no \\
$(+1,+1,-1)$ & 1 &  {\bf \{1,2\} }& {\bf yes} \\
$(+1,-1,+1)$ & -1 & \{1,3\} & no \\
$(+1,-1,-1)$ & -1 & \{1\} & no \\
$(-1,+1,+1)$ & 1 & {\bf \{2,3\}} & {\bf yes} \\
$(-1,+1,-1)$ & -1 & \{2\} & no \\
$(-1,-1,+1)$ & -1 & \{3\} & no \\
$(-1,-1,-1)$ & -1 & $\emptyset $& no \\ \hline
\end{tabular}
\end{center} 
\vskip .08in
Then it  can be verified that $f(x) = \max \big \{ \min\{x_1,x_2\}, \min\{x_2,x_3\}\big \}$.

Let $\mathcal{S}$ be a Sperner family associated to   $f\in\mathcal{M}^*_n$, and for $\Lambda\in \mathcal{S}$ let $y_{\Lambda}\in \{-1,1\}^n$ be given by the bijection described in the proof of Sperner's Lemma. These points  have  the property that $f(x) =1$ for all $x\geq y_{\Lambda}$ and $f(x)=-1$ for all $x<y_{\Lambda}$, and they describe the \emph{boundary} of  the monotone non-constant function $f$. For the constant function $f\equiv 1$ the boundary is $\{-e\}$, and for the constant function $f\equiv -1$ the boundary is  $\{e\}$.

\subsection{Structure of Monotone Functions}\label{sec:structureMonotone1}

In this section we collect some useful facts about monotone functions.

\begin{proposition}\label{prop:monotonePaul} A function $f\in \mathcal{M}_{n+1}$  if and only if the functions $f^{\pm}\in \mathcal{M}_n$  and $f^-\leq f^+$.
\end{proposition}
\begin{proof}  Without loss of generality we will assume we are splitting $f$ according to $x_{n+1}$.

$(\Rightarrow$) If $f\in \mathcal{M}_{n+1}$, then for $x\in \{-1,1\}^n$, we have $f^+(x)=f(x,1)$ and $f^-(x)=f(x,-1)$. If $x,y\in \{-1,1\}^n$ and $x_i\leq y_i$ for all $i\in [n]$
then $f^{\pm}(x)=f(x,\pm 1) \leq f(y,\pm 1)=f^{\pm}(y)$, as in both cases $x_{n+1}=y_{n+1}$, therefore both $f^{\pm}$ are monotone. Furthermore, monotonicity of $f$ implies that  $f(x,-1)\leq f(x,1)$ for all $x\in \{-1,1\}^n$, therefore $f^-\leq f^+$.

($\Leftarrow$) Assume now that $f^{\pm}\in \mathcal{M}_n$ and $f^-\leq f^+$, then $f(x,\pm1) = f^{\pm}(x)\leq  f^{\pm}(y) = f(y,\pm1 ) $ whenever $x_i\leq y_i$ for all $i\in [n]$. The one case that is not covered is when  we consider the points $(x,-1)$ and $(y,1)$ in $\{-1,1\}^{n+1}$, with $x_i\leq y_i$ for all $i\in [n]$. 
Using first the monotonicity of  $f^-$, second that $f^-\leq f^+$, we conclude that for all $x,y\in \{-1,1\}^n$ with $x_i\leq y_i$ for all $i\in [n]$
\[ f(x,-1)=f^-(x) \leq f^-(y)\leq f^+(y)=f(y,1).\]
\end{proof}
 
  Let us flesh out what the condition $g\leq f$ implies in terms of the tribes composing the monotone functions $f$ and $g$. Given $T\subset [n]$, then   $\min T:= \min\{x_m:m\in T\}$.
  \begin{lemma}\label{lem:monotonef<g}
  Given $f,g\in \mathcal{M}_n$  with general tribes expansions
  $$f=\max\{\min{T_1},\dots, \min{T_k}\}, \quad\quad\mbox{and}\quad\quad g=\max\{\min{U_1},\dots, \min{U_j}\},$$ then $g\leq f$ if and only if for all $i \in [j]$ there is $\ell\in [k]$ such that $T_{\ell}\subset U_i$. 
  \end{lemma}
  \begin{proof}
  Note that $g\leq f$ if and only if  $g(x)\leq f(x)$ for all $x\in \{-1,1\}^n$, and this occurs if and only if $f(x)=-1$ implies $g(x)=-1$, and $g(x)=1$ implies $f(x)=1$. Now $g(x)=1$ if and only if  $x_m=1$ for all $m\in U_{i}$ for some $i\in [j]$, similarly for $f$. 
  On the other hand,  $f(x)=-1$   if and only if  $x$ has a component equal to $-1$ in every tribe $T_{\ell}$,  similarly  for  $g$.  
  
 Assume  that $g\leq f$. Given $i\in [j]$, let  $x\in \{-1,1\}^n$ be such that   $x_m=1$ if and only if $m\in U_i$, then $g(x)=1$ and by assumption $f(x)=1$. This implies that there is $\ell\in [k]$ such that $x_m=1$ for all $m\in T_{\ell}$ and therefore $T_{\ell}\subset U_i$. 
 
 Assume now that for all $i \in [j]$ there is $\ell\in [k]$ such that $T_{\ell}\subset U_i$. If $f(x)=-1$ then $x$ has a component equal to $-1$ in every tribe $T_{\ell}$, and by assumption also in every tribe $U_i$, hence $g(x)=-1$. If $g(x)=1$ then there is a tribe $U_i$ such that $x_m=1$ for all $m\in U_i$, by assumption this implies that $x_m=1$ for all $m\in T_{\ell}$ for some $\ell\in [k]$, and hence $f(x)=1$. We conclude that $g\leq f$.
  \end{proof}
  
 The following lemma shows that we can always find a monotone Boolean function such that the $L^1$-norm of its Fourier coefficients can be made arbitrarily large.
\begin{lemma}\label{claim4}
 For all $\lambda >0$ there is $n> 0$ and a   function $f\in \mathcal{M}_n$ such that 
 $$\|\widehat{f}\|_1:= \sum_{S\subset [n]} |\widehat{f}(S)| > \lambda.$$
 \end{lemma}

 \begin{proof}[Proof of Lemma~\ref{claim4}]
 For each $k\geq 1$, let $n=3(k+1)$, and let us divide the $n$ variables into $k+1$ consecutive disjoint blocks with $3$ elements in each block. Let $g_1, g_2, \dots, g_{k+1}$ denote the minimum function in each block, and let $F_{\ell}\in \mathcal{B}_n$ for each $\ell\in[k+1]$ denote the following disjoint tribes function (a monotone function)
 $$F_{\ell}= \max\{g_1,g_2,\dots, g_{\ell}\}.$$
 We can write $F_{k+1}=\max\{ F_k, g_{k+1}\},$ and by  \eqref{eq:max} we have that  $ F_{k+1} = 1 - \frac12 (1-g_{k+1})(1-F_k)$.  Define  $H_k=1-F_k$ and $G_k=1-g_k$. These are not  Boolean functions as their images are in $\{0,2\}$ but we can still do algebra and Fourier analysis with them. We have the recurrence relation:
  \[ H_{k+1}=\frac12 G_{k+1}H_k.\]
  Let $\Lambda^k$ denote the labels of the variables in $F_k$ and $\Lambda_{k+1}$ the labels of the  variables in  $g_{k+1}$. These provide a partition for $[n]$ and, by Remark~\ref{remark2}, we have  that for $S\subset [n]$:
  \[ \widehat{H}_{k+1}(S) = \frac12 \widehat{G}_{k+1}(S\cap \Lambda_{k+1}) \widehat{H}_k(S\cap \Lambda^k),\]
 which implies that
 \[ \sum_{S\subset [n]} \big | \widehat{H}_{k+1}(S)\big | = \frac12 \Big ( \sum_{A\subset \Lambda_{k+1}} \big | \widehat{G}_{k+1}(A)\big | \Big )\Big ( \sum_{B\subset \Lambda^{k}} \big | \widehat{H}_{k}(B)\big | \Big ).\]
 A simple calculation shows that  for $G_{k+1}=1-g_{k+1}=\min\{x_i: i\in\Lambda_k\}$,we have:
 \[ \big | \widehat{G}_{k+1}(A) \big | = \left \{ \begin{array}{cc} 2\big (1-\frac{1}{2^3}\big ) & A=\emptyset\\
  											    \frac{2}{2^3} & A\neq \emptyset \end{array} \right. . \]
 Therefore,
 ${  \sum_{A\subset \Lambda_{k+1}} \big | \widehat{G}_{k+1}(A)\big | = \frac{14}{8} +\frac{2}{2^3}(2^3-1) = \frac{14}{4}>3,}$
 which in turn implies that
 \[ \|\widehat{H}_{k+1}\|_1\geq \frac32 \|\widehat{H}_k\|_1.\]
 Since $\|\widehat{H}_1\|_1>0$, then   $\lim_{k\to\infty} \|\widehat{H}_k\|_1=\infty$, and  therefore  $\lim_{k\to\infty} \|\widehat{F}_k\|_1=\infty$.  Choose $k$ large enough so that $\|\widehat{F}_k\|_1 >\lambda$, let $n=3k$ then $f=F_k\in \mathcal{M}_n$ has the desired property.
  \end{proof}

  Recall that  $\mathcal{M}^*_n$ denotes the  functions $f\in\mathcal{M}_n$ that are not constant. 
The following lemma describes how to produce an $F\in \mathcal{M}^*_{n+1}$ given $F^+\in \mathcal{M}^*_n$, by slightly perturbing $F^+$ at the boundary to produce a monotone function $F^-\leq F^+$.
 
 \begin{lemma}\label{lem:BMN}
  If $f\in \mathcal{M}^*_n$ then there is $y\in \{-1,1\}^n$ such that the function $F\in \mathcal{B}_{n+1}$ defined by $F^+=f$ and $F^- = f-2\delta_y$ belongs to $\mathcal{M}^*_{n+1}$.
 \end{lemma}
 
 \begin{proof}  
 Let $y\in \{-1,1\}^n$ be a point on the boundary of the monotone Boolean function $f$, that is $f(y)=1$ and $f(x)=-1$ for all $x<y$ (by monotonicity it must be that $f(x)=1$ for all $x>y$). The function $g=f-2\delta_y$ is Boolean with $g(y)=-1$ and $g(x)=f(x)$ for all $x\neq y$.  
 By definition, $g(x)\leq f(x)$ for all $x\in \{-1,1\}^n$. To verify monotonicity of $g$, all we need to check is what happens when one of the points is $y$ (since $g(x)=f(x)$ for all $x\neq y$ and $f$ is monotone). Clearly, $g(y)=-1\leq g(x)$ for all $x\in \{-1,1\}^n$, in particular $g(y)\leq g(x)$ whenever $y\leq x$. Suppose now that $x <y$, then  $f(x)=-1$, and  $g(x)=f(x)=-1=g(y)$.  We conclude that $g$ is monotone, and by Proposition~\ref{prop:monotonePaul}, so is $F$, furthermore $F$ is not constant as $F(e,1)=f(e)=1$ and $F(-e,-1) = g(-e)=f(-e)=-1$, so $F\in \mathcal{M}^*_{n+1}$. 
 \end{proof}

For the final result, we need to introduce a little more notation. The function $\beta(f) = \frac12 (1-f)$ is a bijection between the set of Boolean functions and the set of functions on the Hamming cube taking values in $\{0,1\}$ (instead of  $\{-1,1\}$). Formula~\eqref{eq:AND}  for the max function, immediately implies the following:
 \begin{proposition}\label{prop:monotone_product}
  If a Boolean function $f = \emph{max}(g_1, g_2, \dots, g_k)$  for some set of Boolean functions $\{ g_i \}_{i=1}^k$, then $\displaystyle \beta(f) = \prod_{i=1}^k \beta(g_i)
$. In particular, when $f$ is a monotone function, and hence a general tribes function, we have $\displaystyle \beta(f) = \prod_{i=1}^k \beta(\emph{min}(T_i))$. 
 \end{proposition}

\section{An Inductive Approach to FEI}\label{sec:InductiveApproachFEI}

In this section we describe in more detail the  inductive approach to the FEI conjecture,   encoded in Ansatz~\ref{ansatz1} presented in the introduction, and how it leads to Ansatz~\ref{ansatz2}.

\subsection{Binary entropy}\label{sec:binaryEntropy}

We will need some notation and a lemma. Let
\[ \phi(x)=\left\{\begin{array}{cc} x\log ({1}/{x}) & \mbox{if $x>0$} \\
                                                  0 & \mbox{if $x=0$},\end{array}\right.  
  \quad   \mbox{and}  \quad
  \psi(x)=\phi (x)+\phi (1-x) \; \mbox{for $0\leq x\leq 1$},
  \]
  where here  $\log x=\log_2x$.  The function $\psi$, called the \emph{binary entropy},  is continuous, symmetric around $1/2$, and non-negative on $[0,1]$. Note also that $\psi'(x)=\phi'(x)-\phi'(1-x)$ so that $\psi'(1/2)=0$. 
We can rewrite the entropy of $f$ as follows
\[ {\bf H}(f) = \sum_{S\subset  [n]} \phi (|\widehat{f}(S)|^2).\]
The following observation (recorded as a lemma) will allow us to further rewrite (in Appendix~B, Section~\ref{sec:algAcrobatics}) the entropy of $f$ in terms of the binary entropy $\psi$, in a way that is more manageable for us.
\begin{lemma}\label{lem:binary-entropy} Let $x,y\in \R$, then 
 $\displaystyle{\phi(x^2) + \phi(y^2) = \phi(x^2+y^2) + (x^2+y^2)\,\psi\bigg (\frac{x^2}{x^2+y^2}\bigg ).}$
\end{lemma}
\begin{proof}
The proof is a calculation using the properties of the logarithm and is left to the reader.                     
\end{proof}

Some other basic and useful properties of these functions are
\begin{equation}\label{eq:gx/2=gx/2}
\phi \Big (\frac{x}{2} \Big )=\frac12 \phi (x) +\frac{x}{2}, \quad \quad\quad\quad  \psi (x)\leq \sqrt{x(1-x)}.
\end{equation}

\subsection{Inductive step}
Let  $f\in\mathcal{B}_n$, and split $f$  by the variable  $x_m$. 
Recall from~\eqref{lem:induction-influence2} that,
\begin{equation*} 
{\bf I}(f)=\frac12\big [{\bf I}(f^+) + {\bf I}(f^-)\big ] + {\bf I}_m(f).
\end{equation*}
It should be clear that an induction argument to prove the FEI conjecture will work if we can show that for all $n\in \mathbb{N}$ and $f\in \mathcal{B}_n$ there is an $m\in [n]$ such that
\begin{equation}\label{convexity-estimate-entropy1}
\Delta{\bf H}(f):={\bf H}(f) -\frac12 \big [{\bf H}(f^+) + {\bf H}(f^-)\big ] \leq K\,  {\bf I}_m(f).
\end{equation}
If this inequality holds for some class of functions that is closed under splitting, such as the Boolean functions, then the FEI conjecture holds for that class of functions.

We have not been able to prove this for either Boolean or monotone Boolean functions. However, in Section~\ref{sec:nuevas_clases} we identify classes of Boolean functions that will satisfy~\eqref{convexity-estimate-entropy1} at all nodes of a binary tree. We  summarize the discussion in the following ansatz stated in the introduction.

\begin{mythm}{1.1}\label{metaThm1} 
Let $\mathcal{A}$ be a class of Boolean functions and suppose  there is a constant $K>0$ such that for every  $f\in\mathcal{A}$ there is a splitting variable $x_m$ for which:
\begin{itemize}
\item[(i)]  The split functions $f^{\pm}\in \mathcal{A}$, 
 \item[(ii)]  $\Delta {\bf H}(f) \leq K\,{\bf I}_m(f)$.
 \end{itemize}
Then the FEI conjecture holds for $\mathcal{A}$, i.e.  ${\bf H}(f)\leq K\,{\bf I}(f)$  for every $f\in \mathcal{A}$.
\end{mythm}
\begin{proof}
The details of the proof by induction  on the dimension $n$ are left to the reader.
\end{proof}

The next ansatz, also from the introduction, shows that the FEI conjecture holds for any Boolean function $f$ having a stopping binary tree $\T$  where all the functions $f^\rho$  at branching nodes $\rho$ satisfy~\eqref{convexity-estimate-entropy1} (appropriately modified to reflect the splitting variable at each node), and at stopping nodes $\rho$ the FEI conjecture holds.
\begin{mythm}{1.2}\label{ansatz4}
\label{metaThm} Let $f$ be a Boolean function with a stopping tree $\T$, and suppose that  there is $K>0$ such that 
\begin{itemize}
\item[(i)] ${\bf H}(f^\rho)\leq K\, {\bf I}(f^\rho)$ at stopping nodes $\rho\in \sigma_{\T}$.
\item[(ii)] ${\Delta}{\bf H}(f^\rho) \leq K\, {\bf I}_{i_\rho} (f^\rho)$ at branching nodes $\rho\in \beta_{\T}$, 
\end{itemize}
Then  ${\bf H}(f)\leq K\, {\bf I}(f).$
\end{mythm}
\begin{proof}
Let $\T$ be the stopping tree  associated to $f$ satisfying the hypothesis in the theorem.  Recall formula~\eqref{eq:stoppingTreeInfluenceFourier}, for the influence of a function in terms of the influence on stopping nodes of $\T$ plus remainders on each of the branching nodes of $\T$:
\[{\bf I}(f) = \sum_{\rho\in \sigma_{\T}} \frac{1}{2^{k_\rho}}{\bf I}(f^\rho) +  \sum_{\rho\in\beta_{\T}}  \frac{1}{2^{k_\rho}} {\bf I}_{i_\rho}(f^\rho).\]
A similar identity holds replacing the  total influence ${\bf I}$ by the entropy ${\bf H}$ and the individual influence ${\bf I}_{i_{\rho}}$ by $\Delta{\bf H}(f^{\rho})$, by recursively using the formula ${\bf H}(f) = \frac{{\bf H}(f^+)+{\bf H}(f^-)}{2} + \Delta{\bf H}(f)$. Then we can  use hypothesis (i) on the branching nodes, and hypothesis (ii) on the stopping nodes, 
 to conclude that
\begin{eqnarray*}
{\bf H}(f) & =  &  \sum_{\rho\in \sigma_{\T}} \frac{1}{2^{k_\rho}}{\bf H}(f^\rho) +  \sum_{\rho\in\beta_{\T}}  \frac{1}{2^{k_\rho}}\Delta{\bf H}(f^\rho)\\
&\leq &  K\Big (\sum_{\rho\in \sigma_{\T}} \frac{1}{2^{k_\rho}}{\bf I}(f^\rho) +  \sum_{\rho\in\beta_{\T}}  \frac{1}{2^{k_\rho}}{\bf I}_{i_\rho}(f^\rho) \Big )
\, = \,  K\, {\bf I}(f).
\end{eqnarray*}
\end{proof}

After performing some algebraic acrobatics using Lemma~\ref{lem:binary-entropy},  detailed in Appendix~B, Section~\ref{sec:algAcrobatics},  we show that  \eqref{convexity-estimate-entropy1} is equivalent to:
\begin{equation}\label{final-inequality}
\tsup[1]{\Delta}{\bf H}(f):= \sum_{S\subset [n]\setminus\{m\}}\frac12\big (\widehat{f^+}(S)^2+\widehat{f^-}(S)^2\big )\,\psi \bigg (\frac{\frac12\big (\widehat{f^+}(S)-\widehat{f^-}(S)\big )^2}{\widehat{f^+}(S)^2+\widehat{f^-}(S)^2}\bigg )\leq K\, {\bf I}_m(f)
\end{equation}
where $x_m$ is the splitting variable. This yields the following result:

\begin{corollary} \label{coro:Ansatz1-2}
We can replace  hypothesis {\rm (ii)}  in {\rm Ansatz~\ref{metaThm1}}  and  {\rm Ansatz~\ref{ansatz4}} (appropriately adapted to the branching nodes) by~\eqref{final-inequality}.
 \end{corollary}

\begin{corollary}\label{coro2:Ansatz1-2}
We can replace  hypothesis {\rm (ii)}  in {\rm Ansatz~\ref{metaThm1}}  and
 {\rm Ansatz~\ref{ansatz4}} (appropriately adapted to the branching nodes) by 
 \begin{equation}\label{eq:CSDSfalse}
\tsup{\Delta}{\bf H}(f):= \sum_{S\subset [n]\setminus\{m\}} \big |\widehat{f^+}(S)^2-\widehat{f^-}(S)^2\big | \leq K\,{\bf I}_m(f).
 \end{equation}
 \end{corollary}
 \begin{proof}
 It suffices to show that if~\eqref{eq:CSDSfalse} holds then~\eqref{final-inequality} also holds. Indeed, as  $\psi (x)\leq \sqrt{x(1-x)}$, for $x\in[0,1]$, a calculation shows that
 \[ \big (\widehat{f^+}(S)^2+\widehat{f^-}(S)^2\big )\,\psi \left (\frac{\frac12\big (\widehat{f^+}(S)-\widehat{f^-}(S)\big )^2} {\widehat{f^+}(S)^2+\widehat{f^-}(S)^2}\right ) \leq 
 \frac12 \big |\widehat{f^+}(S)^2-\widehat{f^-}(S)^2\big |.\]
 Adding over all $S\subset [n]\setminus\{m\}$, and using ~\eqref{eq:influence-fpm}, gives the desired inequality.
 \end{proof}

Using ~\eqref{eq:influence-fpm}, we can replace the individual influence in any of these inequalities by 
\[
\frac14 \sum_{S\subset [n]\setminus\{m\}} \big (\widehat{f^+}(S)-\widehat{f^-}(S)\big )^2, 
\]
and we will regularly move over and back between these two versions.

In the following sections we will use the preceding methods to prove that the FEI conjecture holds for various classes of Boolean functions. This includes a mix of previously known and new results. We will make use of the following simple lemma about Boolean functions.

\begin{lemma}\label{simple_ineq}
For any $S\subset [n]\setminus\{m\}$, $f\in \mathcal{B}_n$, we have:
\[
\big | \widehat{f^+}(S)^2-\widehat{f^-}(S)^2 \big | \leq 4\,{\bf I}_m(f).
 \]
 where $f$ has been split by $x_m$.
\begin{proof}
\begin{eqnarray*}
\big | \widehat{f^+}(S)^2-\widehat{f^-}(S)^2 \big | &=& | \widehat{f^+}(S) + \widehat{f^-}(S) \big | \big | \widehat{f^+}(S) - \widehat{f^-}(S) \big |\\
&\leq & 2 \big |  \widehat{f^+}(S) - \widehat{f^-}(S) \big | = 2 \big | {\bf E}\left( (f^+ - f^- ) \chi_S \right) \big | \\
&\leq& 4 {\bf E} \left( \frac12 |f^+ - f^- | \right)  = 4 {\bf E} \left(\frac14 |f^+ - f^- |^2 \right) \\
&=& \,{\bf I}_m(f).
\end{eqnarray*}
Here we have use the fact that $\frac12 |f^+ - f^- |= \frac14 |f^+ - f^- |^2$ because $\frac12 |f^+ - f^- |$ can only take the values $0$ and $1$. 
\end{proof}
 \end{lemma}

\section{$\delta$-tribes}\label{sec:(counter)examples}

The class of Boolean functions is closed under splitting by no matter what variable, that is if $f\in \mathcal{B}_n$ then $f^{\pm}\in \mathcal{B}_{n-1}$.
The same holds true for monotone Boolean functions (see Proposition~\ref{prop:monotonePaul}) and for the subclass of monotone Boolean functions given by disjoint tribes. If in either of these categories: Boolean, monotone Boolean, or disjoint tribes, we could show  inequality~\eqref{final-inequality} holds, then we could prove the FEI conjecture is valid  for the corresponding category,  by Ansatz~\ref{ansatz1}. In this section we will show that the choice of splitting variable matters for the validity of the inequality. In Section~\ref{sec:monotone-almost-disjoint} we present an example of a family of $\delta$-tribes (a subset of the monotone functions) for which the wrong choice of splitting variable leads to the failure of the stronger  inequality~\eqref{eq:CSDSfalse}. In Section~\ref{sec:choosing-variable} we show that by choosing a different variable  the stronger inequality~\eqref{eq:CSDSfalse}  holds, and therefore so does  inequality~\eqref{final-inequality}.  

\subsection{Monotone  counterexample?}\label{sec:monotone-almost-disjoint}
 
 We use  Lemma~\ref{claim4} to construct for each $K>0$ a function $F\in \mathcal{M}_{n+1}$ for $n$ large enough  that disproves~\eqref{eq:CSDSfalse} when splitting by the variable $x_{n+1}$. 
 
 \begin{proposition}\label{cor:CSDSfalse}
 For any $K>0$, there is $n>0$ and an $F\in \mathcal{M}_{n+1}$ such that 
 \[ \sum_{S\subset [n]} \big | \widehat{F^+}(S)^2-\widehat{F^-}(S)^2\big |> K \sum_{S\subset [n]} \big | \widehat{F^+}(S)-\widehat{F^-}(S)\big |^2.\]
 \end{proposition}
 \begin{proof}
 For $\lambda >1$, by Lemma~\ref{claim4}, there is $n>0$ and  a monotone function $f\in \mathcal{B}_n$ such that $ \sum_{S\subset [n]} |\widehat{f}(S)| > \lambda$. If $f$ were constant then  $\widehat{f}(S) = 0$ for all $S\neq\emptyset$, and $|\widehat{f}(\emptyset)|=1$, therefore $\sum_{S\subset [n]} |\widehat{f}(S)| =1< \lambda$, which is a contradiction, hence $f \in \mathcal{M}^*_n$. Let  $F\in \mathcal{M}^*_{n+1}$ be the function provided by Lemma~\ref{lem:BMN} (see Figure~\ref{fig:counterexample3}(a)), then for $S\subset [n]$, there is $y\in \{-1,1\}^n$ such that
 \[ \widehat{F}^+(S)=\widehat{f}(S), \hskip 1in \widehat{F}^-(S)= \widehat{f}(S) -2\widehat{\delta}_y(S).\]
 But $|\widehat{\delta}_y(S)|=\frac{1}{2^n}$ for all $S\subset [n]$, thus $\big | \widehat{F^+}(S)-\widehat{F^-}(S)\big |=\frac{1}{2^{n-1}}$ for all $S\subset [n]$. Hence
 \[ \sum_{S\subset [n]} \big | \widehat{F^+}(S)-\widehat{F^-}(S)\big |^2  =  \sum_{S\subset [n]} \left (\frac{1}{2^{n-1}}\right )^2= 2^n
 \frac{1}{2^{2n-2}} = \frac{1}{2^{n-2}}, \]
 and
 \begin{eqnarray*}
  \sum_{S\subset [n]} \big | \widehat{F^+}(S)^2-\widehat{F^-}(S)^2\big | & = &  \sum_{S\subset [n]} \big | \widehat{F^+}(S)-\widehat{F^-}(S)\big | \big | \widehat{F^+}(S)+\widehat{F^-}(S)\big |\\
  &  = & \frac{1}{2^{n-1}} \sum_{S\subset [n]} \big | \widehat{F^+}(S)+\widehat{F^-}(S)\big | 
  \, = \, \frac{1}{2^{n-1}} \sum_{S\subset [n]} |2\widehat{F^+}(S)-2\widehat{\delta}_y(S)| \\
  &  \geq & \frac{1}{2^{n-2}} \sum_{S\subset [n]} \Big ( |\widehat{F^+}(S)|-\frac{1}{2^n}\Big ) 
  \, = \,  \frac{1}{2^{n-2}} \left (\sum_{S\subset [n]} |\widehat{f}(S)| - \sum_{S\subset [n]} \frac{1}{2^n}\right ) \\
  &  \geq  &  \frac{\lambda-1}{2^{n-2}}  \,  = \, (\lambda -1)  \sum_{S\subset [n]} \big | \widehat{F^+}(S)-\widehat{F^-}(S)\big |^2.
 \end{eqnarray*}
 Given $K>0$, choose $\lambda > K+1$, this shows that there is  $n\geq 1$ and  a monotone function $F\in \mathcal{B}_{n+1}$ such that the proposition holds,
 and therefore \eqref{eq:CSDSfalse} is false when splitting by the variable $x_{n+1}$.
 \end{proof}

  This calculation does not disprove inequality~\eqref{final-inequality} for the chosen splitting variable, that would require a separate but similar calculation.

  \subsection{Choosing the right splitting variable}\label{sec:choosing-variable}


  The example $F\in \mathcal{M}_{n+1}$ used in Section~\ref{sec:monotone-almost-disjoint},  for $n=3k$ large enough, has a general tribes decomposition of the form
   \[ F = \max\{\min T_1\dots, \min T_{k-1}, \min S_1, \dots, \min S_{n-3}, \min S_{n+1}\},\] 
  where the $T_{\ell}$ are consecutive, disjoint tribes with three elements and $S_j= \{x_j\}\cup T_k$, by Lemma~\ref{lem:f-2deltay}. 
  Note that the tribe $T_k$ plays a special role. In Section~\ref{sec:monotone-almost-disjoint}, we split $F$ using the variable   $x_{n+1}$ which only belongs to the tribe $S_{n+1}$.  Splitting by any variable from $T_k$, which belongs to all the $S_{\ell}$ tribes,  and  continuing to  split using the other variables in $T_k$,  we obtain the stopping tree depicted in Figure~\ref{fig:counterexample3}(b)  with disjoint tribes $H= \max\{\min T_1,\dots, \min T_{k-1}\}$ and $G=\max\{x_1,\dots, x_{n-3}, x_{n+1}\}$ at all stopping nodes,   and where $F^+_1$  and $F^+_2$ have the same structure as $F$.

We can show that  $\tsup\Delta{\bf H}(F)\leq 38\, {\bf I}_{y_1}(F)$ not just for this specific example, but more generally  for an arbitrary number of  tribes of arbitrary widths. To this end, we introduce a class of  monotone Boolean functions, the  \emph{$\delta$-tribes}  functions,  as follows: 
    
 \begin{definition} 
A function $F\in \mathcal{M}_{n+1}$ is a \emph{$\delta$-tribes} function if 
  \[
    F = \max\{\min T_1\dots, \min T_{k-1}, \min S_1, \dots, \min S_{n-u}, \min S_{n+1}\}, 
  \]
where $\{T_{\ell}: \ell\in [k]\}$ is a set of consecutive, disjoint tribes covering the first $n$ variables,  $u$ is the number of elements in the last tribe $T_k$, and $S_j= \{x_j\}\cup T_k$.
\end{definition}
   
Note that when splitting a $\delta$-tribes function $F$ by the variable $x_{n+1}$ then $F^+=f= \max\{\min T_1\dots, \min T_{k-1}, \min T_k\}$ and $F^-=f-2\delta_y$, 
where the $i$-th element of $y$ is $1$ if $i\in T_k$ and is $-1$ otherwise (see Lemma~\ref{lem:f-2deltay}), justifying the name $\delta$-tribes.

 \begin{theorem}
 The FEI conjecture holds for $\delta$-tribes functions.
 \end{theorem}
 
    \begin{figure}
    \hskip .5in
   \begin{tikzpicture}
\Tree
[.$F,x_{n+1}$    
    [.$f-2\delta_y$ ]
    [.$f$ ]
    ]
\end{tikzpicture}
 \hskip 1.5in 
\begin{tikzpicture}
\Tree
[.$F,y_1$    
    [.$H$ ]
    [.$F_1^+,y_2$ 
    \edge[]; {$H$}
    \edge[]; [.$F_2^+,y_3$
                    \edge[]; {$H$}
                    \edge[]; {$G$}
                    ]
         ]
    ]

\end{tikzpicture}
\caption{\footnotesize{(a) Split w.r.t to $x_{n+1}$. $\quad\quad\quad\quad$ (b) Stopping tree split w.r.t  $T_k=\{y_1,y_2,y_3\}$.}} \label{fig:counterexample3}
 \end{figure}
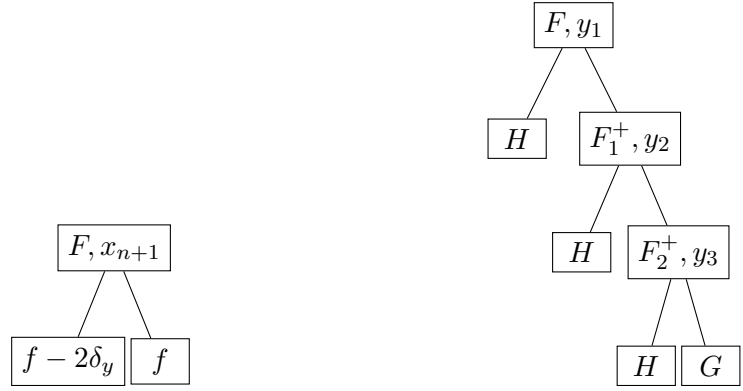
 
\begin{proof}
We will show that any $\delta$-tribes function $F$ can be represented by a stopping binary tree such that at the stopping nodes we have disjoint tribes functions and at the branching nodes we have $\delta$-tribes functions for whom inequality~\eqref{eq:CSDSfalse} holds.  As the FEI conjecture holds for disjoint tribes, by Ansatz~\ref{ansatz4}  (more precisely by Corollary~\ref{coro2:Ansatz1-2}) the FEI conjecture will hold for the $\delta$-tribes functions.

We label the elements of $T_k$ as $\{ y_1, \dots, y_u\}$ and we start building the tree by setting  $F^+_0 : =F$ and $S^0_j=S_j$ for each $j$. Now split by $y_1$, the first variable in $T_k$, and apply   Proposition~\ref{prop:monotoneFgivenfandg}(ii) in the Appendix, to obtain two new functions 
 \begin{eqnarray*}
 F^+_1 & = & \max\left \{\min T_1, \dots, \min T_{k-1}, \min S_1^1, \dots,  \min S_{n-u}^1, \min S_{n +1}^1\right \} ,\\
 F^-_1 & = & \max\{ \min T_1, \dots, \min T_{k-1}\} := H.
 \end{eqnarray*}
where $S^1_j=S^0_j\setminus \{y_1\}$. Note that $F^-_1=H$ is a disjoint  tribes function and that $F^+_1$ is a $\delta$-tribes function, except that the $S^1_j$ tribes have one fewer element (the splitting variable $y_1$ has been removed) than the $S^0_j$ tribes. Continuing this process, after $u$ steps we end up with $S^{u}_j=S^{u-1}_j \setminus\{y_u\}= S^0_j\setminus T_k= \{x_j\}$, to get  the split functions $F^{-}_u=H$ and 
\begin{eqnarray*}
 F^{+}_u & = & \max\left \{\min T_1,\dots, \min T_{k-1}, x_1 , \dots,  x_{n-u}, x_{n+1}\right \} \\
   & = &  \max\left \{x_1, \dots, x_{n-u}, x_{n+1}\right \} := G.
 \end{eqnarray*}
The last inequality follows because the tribes $T_{\ell}$ for $\ell\in [k-1]$ are redundant. At this point,  both $F^{\pm}_u$ are disjoint tribes and we stop the process.  Figure~\ref{fig:counterexample3}(b) illustrates the case $u=3$. As $G$ and $H$ are disjoint tribes, it remains to show that  inequality~\eqref{eq:CSDSfalse} holds when we split a $\delta$-tribes function by one of the variables in $T_k$.  

Without loss of generality, we will show that if we split $F$ by the first variable in $T_k$, then there is $K>0$ independent of the parameters $k$, $w$, and $u$, such that \eqref{eq:CSDSfalse} holds:
\[ 
\sum_{S\subset [n+1]^*}   \big | \widehat{F^+}(S)^2-\widehat{F^-}(S)^2\big | \leq  {K }\,{\bf I}_{n-u+1}(F).
\] 
Here, $ [n+1]^* =  [n+1]\setminus \{n-u+1\}$, where $n-u+1$ is the index of the first variable in $T_k$.

 Starting at $F$ and working downwards, noting that $F^-_j = H$ for all $j \in [u]$\footnote{Technically each of these functions is defined in a different Hamming cube, but since they  only depend on the variables $x_{\ell}$ with $\ell\in [n-u]$, we can view them as Boolean functions in any cube $\{-1,1\}^N$ for $N\geq n-u$.}, and repeatedly using formula~\eqref{eq:fsplitxm},   we get
 \begin{eqnarray*}
 F^+_0 & = &  \frac{F^+_1 + H}{2} + y_1 \frac{F^+_1 - H}{2} \\
 \vdots & & \hskip .65in \vdots \\
 F^+_{u-1} & = &  \frac{F^+_{u} +H}{2} + y_u \frac{F^+_{u} - H}{2} 
 \end{eqnarray*}
Recall that $F^+_u = G = \max\left \{x_1, \dots, x_{n-u}, x_{n+1}\right \}$, therefore
\[F^+_{u-1} = \frac{G+H}{2} + y_u\frac{G-H}{2} = H+ \frac12 (G-H)(1+y_u).\]
 Substituting the last equation into the corresponding one for $F^+_{u-2}$, we get 
 \begin{eqnarray*}
 F^+_{u-2}  & = & H+\frac14 (G-H)(1+y_{u-1})(1+y_u)
 \end{eqnarray*}
 The pattern is now clear, as can be verified by induction, for all $i\in [u]$:
\[
  F^+_{u-i} =  H + \frac{1}{2^i} (G-H)\prod_{j=u-i+1}^u(1+y_j)  := H + \frac{1}{2^i} (G-H) P_{u-i}.
\]
The function $F^- = H$ depends only on the variables $\{x_1,\dots, x_{n-u}\}$, thus if $S$ contains any of the labels of the remaining variables $\{x_{n-u+2}, \dots, x_{n+1}\}$, we have $\widehat{F^-}(S) =0$. Therefore,
  \begin{equation}\label{case1}
   \sum_{\substack{S\in [n+1]^* \\  S\cap S_1\neq \emptyset}}  \big | \widehat{F^+}(S)^2-\widehat{F^-}(S)^2\big | 
= \sum_{\substack{S\in [n+1]^* \\  S\cap S_1\neq \emptyset}}  \big | \widehat{F^+}(S)-\widehat{F^-}(S)\big |^2 \leq 4  \,{\bf I}_{n-u+1}(F),
   \end{equation}
  where  $S_1=\{n-u+2, \dots,  n+1\}$.

By Lemma~\ref{simple_ineq}, we have:
  \begin{equation}\label{case2}
  \big | \widehat{F^+}(\emptyset)^2-\widehat{F^-}(\emptyset)^2\big |  \leq 4  \,{\bf I}_{n-u+1}(F).
   \end{equation}

We are left to consider the sum over non-empty sets $S\subset [n+1]^*$ where $S\cap S_1=\emptyset$. This is the same as considering the sum over non-empty $S\subset [n-u]$. 
For the remainder of the proof, sums will be over non-empty  $S\subset [n-u]$, unless stated otherwise. We have
\begin{eqnarray*}
 F^+ & = & F^+_1  = H + \frac{1}{2^{u-1}} (G-H)P_1,\\
 F^-& = & H \; = \; \max\{\min T_1,\dots, \min T_{k-1}\},\\
 G & = & \max\{x_1,\dots, x_{n-u},x_{n+1}\}.
 \end{eqnarray*}
As $P_1$ depends only on the variables $\{y_2,\dots, y_u\}$,  it is independent of $x_i$ for $i \in [n-u]$. Furthermore, $\widehat{P_1}(\emptyset)= 1$, so it follows that for $S \subset [n-u]$:
 \[
\widehat{F}^+(S)= \widehat{H}(S) + \frac{1}{2^{u-1}}\Big (\widehat{G}(S) - \widehat{H}(S)\Big ).
\]

   Since $F$ is monotone and $\widehat{G}(\emptyset )= 1-\frac{1}{2^{n-u}}$, by Example~\ref{eg:AND}, we have
   \begin{eqnarray}\label{eq:Iws+1}
   {\bf I}_{n-u+1}(F) & = & \frac{\widehat{F}^+(\emptyset) -\widehat{F}^-(\emptyset) }{2} \, = \,\frac{ \widehat{G}(\emptyset) -\widehat{H}(\emptyset)}{2^{u}} 
             \, = \, \frac{1}{2^{u}} \Big ( 1-\frac{1}{2^{n-u}} - \widehat{H}(\emptyset) \Big ). 
   \end{eqnarray}

   We  have  formula~\eqref{eq:DisjointTribesEmptysetCoeff} for the Fourier coefficient of the empty set of a disjoint tribes, namely
     \[ 
\widehat{H}(\emptyset)        =  1- 2\prod_{\ell=1}^{k-1} \Big ( 1-\frac{1}{2^{w_{\ell}}}\Big ) ,
\]
where $w_{\ell}$ is $\#T_\ell$. As  $\frac{1}{2^{w}}\leq 1-\frac{1}{2^{w}}$ and $n-u=w_1+\dots + w_{k-1}\geq 1$, we get 
\begin{equation}\label{lb1}
 \frac{2}{2^{n-u}} \leq 1-\widehat{H}(\emptyset ),
\end{equation}
\begin{equation}\label{lb2}
\frac{1}{2^{u+1}} \Big ( 1 - \widehat{H}(\emptyset) \Big ) \leq {\bf I}_{n-u+1}(F).
\end{equation}

Next, we have $|a^2 - b^2| \leq |a-b|^2 + 2|(a-b)b|$ for any real numbers $a$ and $b$, thus
\[
  \sum \big | \widehat{F^+}(S)^2-\widehat{F^-}(S)^2\big | \leq 4\, {\bf I}_{n-u+1}(F)  +  
\sum \frac{2}{2^{u-1}}  \big | \widehat{G}(S)-\widehat{H}(S)  \big |  \big |\widehat{H}(S) \big |
\]
The triangle inequality implies that the second term on the right is bounded by:
\[
\sum \frac{2}{2^{u-1}}  \big | \widehat{G}(S) \big |  \big |\widehat{H}(S) \big |  
+   \sum \frac{2}{2^{u-1}}   \big |\widehat{H}(S) \big |^2.
\]
The second term here is easy to control as:
\[
\sum \big |\widehat{H}(S) \big |^2 = 1-\widehat{H}(\emptyset)^2 =(1+\widehat{H}(\emptyset)) (1-\widehat{H}(\emptyset)) \leq 2(1 - \widehat{H}(\emptyset))
\]
As $|\widehat{G}(S)| = \frac{1}{2^{n-u}}$ for non-empty $S$, and using~\eqref{lb1}, we have:
\begin{eqnarray*}
\sum 2 \big | \widehat{G}(S) \big |  \big |\widehat{H}(S) \big | 
& \leq & \sum \Big (\big | \widehat{G}(S) \big |^2 +  \big |\widehat{H}(S) \big |^2   \Big )
=  \frac{2^{n-u}-1}{2^{2(n-u)}} +\sum \big |\widehat{H}(S) \big |^2\\
&\leq&  \frac{1}{2^{(n-u)}} + 2(1 - \widehat{H}(\emptyset)) \leq \frac52 (1 - \widehat{H}(\emptyset))
\end{eqnarray*}

Combining these estimates, including~\eqref{lb2} we get:
\begin{equation}\label{case3}
  \sum_{S\subset [n-u]: S\neq\emptyset}  \big | \widehat{F^+}(S)^2-\widehat{F^-}(S)^2\big | \leq 30\, {\bf I}_{n-u+1}(F) 
\end{equation}

Collecting all cases in estimates~\eqref{case1}, \eqref{case2}, and~\eqref{case3}, we  conclude that 
\[ \sum_{S\subset [n+1]\setminus \{n-u+1\}}  \big | \widehat{F^+}(S)^2-\widehat{F^-}(S)^2\big |   \leq  K\, {\bf I}_{n-u+1}(f),\]
where $K=38$.

  \end{proof}

  \section{Some new classes of Boolean functions}\label{sec:nuevas_clases}\label{sec:NewClassesBooleanFunctions}
  
We say that a collection of functions are \emph{separated} if their sets of defining variables are pairwise disjoint, and that two functions $f$ and $g$ are \emph{semi-separated} if $f= pq$ and $g= pr$, where $p$, $q$, and $r$ are separated. If the functions  $p$, $q$, and $r$ all belong to some class of functions $\mathcal A$, then we say that $f$ and $g$ are semi-separated by class $\mathcal A$. A Boolean function $f$ is said to be  \emph{semi-separated} if there is a splitting variable such that  $f^{\pm}$ are semi-separated.

It will be convenient to consider the set of functions $\mathcal{G}_n$ on the Hamming cube that take values in $\{-1,0,1\}$. This includes the Boolean functions, their derivatives, delta functions,  and the functions $\beta(f)$ appearing in Proposition~\ref{prop:monotone_product}. The key property of this class is that $g^2 = |g|$ for any function $g\in \mathcal{G}_n$. The next proposition, in combination with Corollary~\ref{coro2:Ansatz1-2}, will be useful for constructing classes of functions that satisfy the FEI conjecture, 
 \begin{proposition}\label{prop:control_semi_sep}
If  $g_1$ and $g_2$  are semi-separated by  $\mathcal{G}_{n}$, then
\[
\sum_{S\subset [n]} \big |\widehat{g_2}(S)^2-\widehat{g_1}(S)^2\big | 
\leq 4 \sum_{S\subset [n]}\big (\widehat{g_2}(S)-\widehat{g_1}(S)\big )^2.
\]
 \end{proposition}
 \begin{proof}
By hypothesis we have $g_1 = pr$ and $g_2=pq$ for separated functions $p,q,r$ in $\mathcal{G}_n$. Let $\Lambda_p, \Lambda_q, \Lambda_r \subset [n]$ be the disjoint, index sets of the variables that define these functions.  

First note that if $S\setminus (\Lambda_p\cup \Lambda_q\cup \Lambda_r)\neq \emptyset$ then $\widehat{g_2}(S)=\widehat{g_1}(S)=0$. So, the only sets that contribute to the sum are  $S\subset \Lambda_p\cup \Lambda_q\cup \Lambda_r$. There are three  cases to consider.

\noindent{\bf Case 1 ($S\cap \Lambda_q\neq \emptyset$).} Then
\[ \widehat{g_1}(S)=\mathbb{E}(p r\chi_S)=\mathbb{E}(p r\chi_{S\setminus\Lambda_q})\mathbb{E}(\chi_{S\cap\Lambda_q})=0.\]
As a result,
\[\sum_{S\in [n]: S\cap \Lambda_q\neq \emptyset} \big  |\widehat{g_2}(S)^2-\widehat{g_1}(S)^2\big | 
\leq  \sum_{S\subset [n]} \big (\widehat{g_2}(S)-\widehat{g_1}(S)\big )^2.
\]

\noindent{\bf Case 2 ($S\cap \Lambda_r\neq \emptyset$).} This is the same as Case 1.

\noindent{\bf Case 3 ($S\subset  \Lambda_p$).} 
In this case, 
\begin{eqnarray*}
 \widehat{g_2}(S) & = & \widehat{pq}(S) \, = \, {\bf E}(pq\chi_S) \, = \,{\bf E}(p\chi_S){\bf E}(q) \, = \, \widehat{p}(S) \, \overline{q}, \\
 \widehat{g_1}(S) & = &\widehat{pr}(S) \, = \, {\bf E}(pr\chi_S) \, = \,{\bf E}(p\chi_S){\bf E}(r) \, = \, \widehat{p}(S) \, \overline{r},
\end{eqnarray*}
where $\overline{r}={\bf E}(r)$. Using this and Plancherel,
\[
\sum_{S\subset [n]:S\subset  \Lambda_p }\big |\widehat{g_2}(S)^2-\widehat{g_1}(S)^2\big  | =	 |\overline{q}^2-\overline{r}^2|  \sum_{S\subset [n]:S\subset  \Lambda_p }\widehat{p}(S)^2 \leq {\bf E}(p^2) \, |\overline{q}^2-\overline{r}^2|.\\
\]
On the other hand, using Plancherel again, we have that 
\begin{eqnarray*}
\sum_{S\subset [n]} \big (\widehat{g_2}(S)-\widehat{g_1}(S)\big  )^2 & = &  {\bf E}\big ((g_2-g_1)^2\big )\, = \, {\bf E}\big (p^2(q-r)^2\big ) \\ 
& = & {\bf E}(p^2) \, {\bf E}\big ((q-r)^2\big)  \, = \,  {\bf E}(p^2) \, {\bf E}\big (q^2+r^2-2rq \big ). 
\end{eqnarray*}
As $q,r \in \mathcal{G}_n$,  $q^2=|q|$ and $r^2=|r|$, hence 
\begin{eqnarray*}
{\bf E}\big (q^2+r^2-2rq \big) &=& {\bf E}\big (|q|+|r|-2rq \big ) =  {\bf E}(|q|) +  {\bf E}(|r|) - 2 {\bf E}(q)  {\bf E}(r)\\
&\geq &  |{\bf E}(q)| + |{\bf E}(r)|- 2 |{\bf E}(q)| | {\bf E}(r)| =  |\overline{q}| +  |\overline{r}| - 2  |\overline{q}| \,| \overline{r}|
\end{eqnarray*}
Next, we observe that for any real numbers $0\leq a, b \leq1$,  we have  $|a^2-b^2|\leq 2 (a+b -2ab)$. Applying this with $a = |\overline{q}|\leq 1$ and $b = |\overline{r}|\leq 1$,  and combining the various estimates,we obtain:
\[
\sum_{S\subset [n]:S\subset  \Lambda_p }\big |\widehat{g_2}(S)^2-\widehat{g_1}(S)^2\big  |  \leq  2 \sum_{S\subset [n]} \big (\widehat{g_2}(S)-\widehat{g_1}(S)\big  )^2.
\]
Collecting Cases 1-3 we conclude that
\[
\sum_{S\subset [n]} \big |\widehat{g_2}(S)^2-\widehat{g_1}(S)^2\big | \leq  4 \sum_{S\subset [n]} \big (\widehat{g_2}(S)-\widehat{g_1}(S)\big  )^2.
\]
  \end{proof}

  \begin{definition} 
A Boolean function satisfies the \emph{semi-separation property} if there is a stopping binary tree where at every branching node the branching functions are semi-separated by Boolean functions, and at stopping nodes the functions satisfy the FEI conjecture. 
 \end{definition}
The following theorem is an immediate consequence of this definition, the preceding proposition and Corollary~\ref{coro2:Ansatz1-2}.
  \begin{theorem}  
Boolean functions that satisfy the semi-separation property satisfy the FEI conjecture.  
\end{theorem}
 
Non-trivial examples of Boolean functions satisfying  the semi-separation property can be constructed starting from the stopping nodes and building up using Lemma~\ref{lem:semi-separated} below. We present two different classes of examples and will use the convenience function defined by $ \gamma(r,q;t) := \frac{q+r}{2} +t\,\frac{q-r}{2}$.

  \begin{lemma}\label{lem:semi-separated}  
Assume that $p,q,r$ are separated functions in $\mathcal{B}_n$. Set 
 $$ F(x,x_{n+1}) = p(x) \,  \gamma \big (r(x), q(x);x_{n+1}\big ) .$$
Then $F$ and $\gamma$ belong to $\mathcal{B}_{n+1}$, $F^+=pq$, $F^-=pr$, and $F$ is separated by $p$ and $\gamma$.
\end{lemma}
The proof is a straightforward exercise. 

 \begin{example}\label{eg:ex1}  Starting with a Boolean function  $p$ that satisfies  the FEI conjecture, pairs of stopping nodes of the form $(-p, p)$ and  single stopping nodes of the form $p$, we use Lemma~\ref{lem:semi-separated}  and different variables (say $x_1, x_2,\dots$, necessarily disjoint from the variables in $p$) for each pair of nodes at all levels to go from the stopping nodes to the root. See Figure~\ref{fig:Paul-separation1}. 
\end{example}
In Example~\ref{eg:ex1}, the variables of $p$ will appear in the full binary tree at least as many times as  there are stopping nodes, and we can make this arbitrarily large. 
 
\begin{figure}
\begin{tikzpicture}
\Tree
[.$p\,\gamma \big (1,\gamma (x_1,1;x_2);x_3\big )$    
    [.$p$ ]
    [.$p\,\gamma (x_1,1;x_2)$ 
    \edge[]; [.$p\,x_1$
                    \edge[]; {$-p$}
                    \edge[]; {$p$}
                    ]
    \edge[]; {$p$}
         ]
    ]
\end{tikzpicture}
\hskip .5in
\begin{tikzpicture}
 \Tree
 [.$x_1\,\gamma\big (r_1,\gamma (r_2,r_3;x_2);x_3\big )$
    [.$x_1r_1$ ]
    [.$x_1\,\gamma(r_2,r_3;x_2)$
    \edge[]; [.$x_1r_2$
       		   \edge[]; {$-r_2$}
		   \edge[]; {$r_2$}
		   ]
   \edge[]; {$x_1r_3$}
       ]
   ]    	
 \end{tikzpicture}
\caption{\footnotesize{Trees illustrating Example~\ref{eg:ex1} (left)  and Example~\ref{eg:ex2} (right).}} \label{fig:Paul-separation1}
\end{figure}
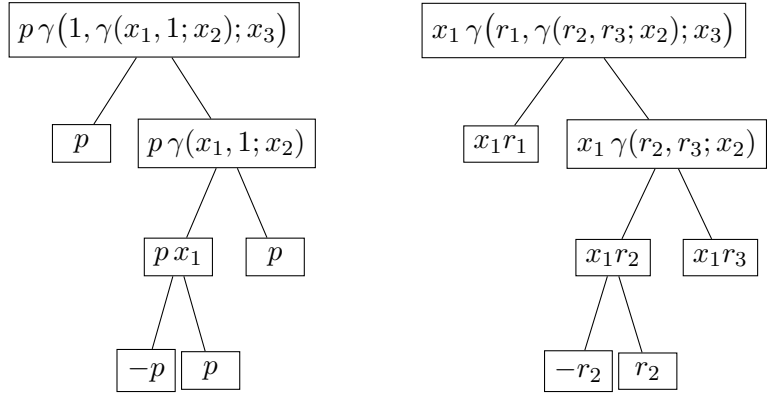
 
 \begin{example}\label{eg:ex2}
Alternatively, we can  start with pairs of stopping nodes  of the form $(-r_i, r_i)$, and single functions of the form $x_1r_j$ where the functions $\{r_i\}$ are  disjointly supported and $x_1$ is a new variable.   To move upwards from the stopping nodes,  we use  $x_1$ for all stopping pairs $(-r_i,r_i)$, for every other node we will use different new variables at each node until we reach the root. Here we are assuming that the functions at the stopping nodes, $r_i$ and $x_1r_j$, all satisfy the FEI conjecture.
See Figure~\ref{fig:Paul-separation1}.
\end{example}

Next, we define a notion of separation for monotone functions that is closely related to the notion of semi-separated functions. We say that a monotone function $f\in \mathcal{B}_n$ is \emph{tribe separated by $x_m$} if the tribes that contain $x_m$ are disjoint from the tribes that do not contain $x_m$. Note that disjoint tribes functions are tribe separated by any of their variables.  

It is not hard to construct examples of monotone functions that are not disjoint tribes but that are tribe separated. For example,  $ f \in \mathcal{B}_{7}$ given by tribes 
 \[
T_1=\{x_1,x_2\}, \; T_2=\{x_2,x_3\}, \; T_3=\{x_4,x_5,x_7\}, \;T_4=\{x_5,x_6,x_7\}
\]
is tribe separated by $x_5$. 

\begin{definition}
A monotone Boolean function $f$ satisfies the \emph{tribe separation property}  if there is a stopping binary tree where at every branching node the function is tribe separated by the splitting variable, and at stopping nodes the functions satisfy the FEI conjecture. 
\end{definition}
 
 \begin{theorem}\label{thm:tribe_separated} 
Monotone Boolean functions that satisfy the tribe separation property satisfy the FEI conjecture.
\end{theorem}
\begin{proof} 
By Corollary~\ref{coro2:Ansatz1-2}, it suffices to show that there is $K>0$ such that
\[
 \sum_{S\subset [n] \setminus\{m\}} \big |\widehat{f^+}(S)^2 - \widehat{f^-}(S)^2 \big | 
\leq K\, {\bf I}_m(f)
\]
at each branching node, where $x_m$ is the splitting variable. Without loss of generality we can assume that $m=n$ and, from Lemma~\ref{simple_ineq}, we can replace the sum by the sum over non-empty sets~$S$. 

For Proposition~\ref{prop:monotone_product}, we defined the operator $\beta(h) = \frac12 (1-h)$, which has some good properties with respect to monotone functions. Observe that $\widehat{\beta(h)}(S) =-\frac12 \widehat{h}(S)$ for non-empty $S$ and that $  {\bf I}_m(\beta(h)) = \frac14 {\bf I}_m(h) $. Therefore, setting $g^+ = \beta(f^+)$ and $g^- = \beta(f^-)$, and replacing the influence by its Fourier form, we see that the inequality above is equivalent to:
  \begin{equation}\label{eq:base_eq}
 \sum_{S\subset [n-1]: S\neq \emptyset} \big |\widehat{g^+}(S)^2 - \widehat{g^-}(S)^2 \big | 
\leq K\, \ \sum_{S\subset [n-1]} \big |\widehat{g^+}(S)-\widehat{g^-}(S)\big |^2.
  \end{equation}

By hypothesis, $f$ is tribe separated by $x_n$. This means that the tribes $T$ not containing $x_n$ are disjoint from the tribes $U$ containing $x_n$. When we split $f$ by $x_n$ we get
\[ 
f^- =\max\{ \min T\}, \quad\quad f^+=\max\{ \min T; \min \big (U\setminus\{x_n\}\big )\}.
\]
Setting $ h = \max\{ \min \big (U\setminus\{x_n\} \big)$, we have $f^+=\max\{f^-, h\}$ where, most importantly, the functions $f^-$ and $h$ are separated. From Proposition~\ref{prop:monotone_product},  
\[
g+ = \beta(f^+) = \beta(f^-) \beta(h) = g^- \beta(h),
\]
implying that $g-$ and $g+$ are separated by $\mathcal{G}_{n-1}$, with $p=g^-$, $q=1$ and $r=\beta(h)$. Applying Proposition~\ref{prop:control_semi_sep}
 yields~\eqref{eq:base_eq} and completes the proof with constant $K=8$.
 \end{proof}  
 
  \section{Appendix A: More on the structure of monotone functions}\label{sec:appendix}
 In this Appendix, we want to understand the general tribes structure of a monotone function $F$ given the general tribes structure of $F^{\pm}$ and viceversa.  
      In fact, we will completely describe in terms of general tribes the monotone functions $F\in \mathcal{B}_{n+1}$ given the general tribes decompositions  of $F^+=f$ and $F^-=g$ where both   $f,g \in  \mathcal{B}_n$  are monotone and $g\leq f$, according to Proposition~\ref{prop:monotonePaul}.  Afterwards we will  specialize to the particular case when $g=f-2\delta_y$, for any $y\in\{-1,1\}^n$ in the boundary of $f$,  that ensures $g$ is monotone (clearly $g\leq f$), as described in Lemma~\ref{lem:BMN}.

  \begin{proposition}\label{prop:monotoneFgivenfandg} Let $m\in [n+1]$.
  
  {\rm (i)}  Suppose  $f$ and $g$ are monotone Boolean functions 
   depending on the variables labeled by $[n+1]\setminus\{m\}$, with   general tribes representation 
  $f=\max\{\min{T_1},\dots, \min{T_k}\}$, and $g=\max\{\min{U_1},\dots, \min{U_j}\},$ such that $g\leq f$. Let $F\in \mathcal{B}_{n+1}$ be defined as
  \begin{equation}\label{eq:FgivenFpm}
  F=\max\left \{\min\{T_1\cup \{x_m\}\}, \dots, \min\{T_k\cup \{x_m\}\},\min U_1, \dots, \min U_j\right \}.
  \end{equation}
  If $F^{\pm}$ denote the split functions by  the variable $x_m$, then $F^+=f$ and $F^-=g$. 
  
  {\rm (ii)} Suppose $F\in \mathcal{B}_{n+1}$  is a monotone function with general tribes representation given by $F=\max\left\{ \min S_1, \dots , \min S_p\right \}$. If $F^{\pm}$ denote
  the split functions  by the variable $x_m$, then 
  $$F^+=\max\{\min{T_1},\dots, \min{T_p}\} \quad  \mbox{and} \quad  F^-=\max\{\min S_q: m\notin S_q\},$$ 
  where  $T_q= S_q\setminus\{x_m\}$  for each $q\in [p]$.
   \end{proposition}

  \begin{proof}
{\rm (i)}  Recall that $F^{\pm}(x_1,\dots,x_{m-1},x_{m+1},\dots,x_{n+1})=F(x^{(m\to \pm 1)})$. Suffices to observe that when $x_m=1$ then $\min\{T_{\ell}\cup \{x_m\}\}(x)=\min T_{\ell}(x)$, and therefore in this case
  \[ F^+= \max\{ \min T_1,\dots, \min T_k,\min U_1,\dots, \min U_j\}.\]
  But by Lemma~\ref{lem:monotonef<g}, for all $i\in [j]$ there is $\ell \in [k]$ such that $T_{\ell}\subset U_i$, which means that the tribes $U_i$ are redundant, and we conclude that when $x_m=1$ then
  $F^+=\max\{\min T_1,\dots, \min T_k\}$, and thus $F^+=f$.
  When $x_m=-1$ then $\min\{T_{\ell}\cup\{x_m\}\}(x)=-1$, and therefore  $F^-= \max\{\min U_1,\dots, \min U_j\}$, 
  thus $F^-=g$. This finishes the proof of {\rm (i)}. 
  
  {\rm (ii)} Let $A=\{q\in [p]: m\notin S_q\}$ (the labels of of the tribes in $F$ that do not contain $x_m$), then $[p]\setminus A=\{q\in [p]: m\in S_q\}$ (the labels of the tribes in $F$ that contain $x_m$).
  
  When $x_m=1$ then $\min S_q = \min (S_q\setminus\{x_m\})=\min T_q$ for all $q\in [p]$,  and when $x_m=-1$ then $\min S_q= -1$ for all $q\in [p]\setminus A$. Therefore 
  $$F^+=\max\left \{ \min T_q : q\in [p] \right \},\quad \quad F^-=\max \{ \min S_q: q\in A\}.$$
  This finishes the proof of {\rm (ii)}.
  \end{proof}
  
  Note that there is no conflict between  {\rm (ii)} and {\rm (i)} in Proposition~\ref{prop:monotoneFgivenfandg}. Indeed,  since $T_q=S_q$ for all $q\in A$, then
  $ f=F^+ = \max\left \{ \min \{T_q : q\in [p]\setminus A\}, \min \{S_q: q\in A\} \right \}$, and $ g= F^-= \max \{ S_q: q\in A\}$.
  Therefore, when we use~\eqref{eq:FgivenFpm} to reconstruct $F$, we get
  \[  \max\left\{   \min \{T_q\cup \{m\} : q\in [p]\setminus A\}, \min \{S_q: q\in A\} \right \} 
     =  \max\{ S_q: q\in [p]\} = F.\]

  We can now use Proposition~\ref{prop:monotoneFgivenfandg} to prove the following lemma.
   
   \begin{lemma}\label{lem:f-2deltay}
     Given   $f=\max\{\min T_1,\dots, \min T_k\} \in \mathcal{M}_n$, with disjoint  tribes $\{T_{\ell}: \ell \in [k]\}$.  Let $g=f-2\delta_y$, where $y$ is a boundary point of $f$, which without loss of generality  we will assume corresponds to the tribe $T_k$.
  Then 
  \[ g=\max\big \{\min \{T_m: m\in [k-1]\},  \min \{S_j: j\in [n]\setminus \Lambda_k\} \big \},\]
  where $S_j= T_k\cup\{x_j\}$ and $\Lambda_k$ are the labels of the variables in $T_k$.
  Furthermore, the  function $F\in \mathcal{B}_{n+1}$   given by
  \[ F= \max \big \{\min T_{\ell }: \ell \in [k-1]\}, \min \{S_j: j\in [n+1]\setminus \Lambda_k\} \big\}, \]
  is such that   when splitting by $x_{n+1}$, we have that $F^+=f$ and $F^-=g$.  
  \end{lemma}
 In the proof of Lemma~\ref{lem:BMN}, we showed that for any $y$ in the boundary of $f$,  the function $g$ is  monotone,  non-constant, and $g\leq f$, hence $g$ must have a representation as a general tribes function by Sperner's Lemma. Lemma~\ref{lem:f-2deltay} provides such representation for $g$ and for $F$. 
\begin{proof}
Comparing to Sperner's Lemma, what we are doing here to go from $f$ to $g=f-2\delta_y$ is removing the minimal set $\Lambda_y$  determined by the boundary point $y$   from  $\mathcal{S}_f$, the Sperner family associated to $f$. 
 The set $\Lambda_y=\{i\in [n]: y_i=1\}$, is clearly a minimal set for $f$ but not for $g$, as $g(y)=-1$.  
 To get a tribes representation for $g$, which may  not be minimal, we  have to  add to $\mathcal{S}_f\setminus \{\Lambda_y\}$ the  sets  $\Lambda^j_y=\Lambda_y \cup \{j\}$ for each $j\notin \Lambda_y$.   Since we assumed $\Lambda_y=\Lambda_k$, then $\Lambda_y^j=\Lambda_k\cup\{j\}$, and we conclude that
 $g=\max\big \{\min \{T_m: m\in [k-1]\},  \min \{S_j: j\in [n]\setminus \Lambda_k\} \big \}$. 
 
    We can now apply  Proposition~\ref{prop:monotoneFgivenfandg} to $f$ and $g$ to conclude that 
   \begin{eqnarray*}
   F & = & \max\left\{ \min\{T_{\ell}\cup\{x_{n+1}\}: \ell\in [k]\}, \min \{T_{\ell}:\ell\in [k-1]\}, \min \{S_j: j\in [n]\setminus \Lambda_k\}\right \} \\
   &= & \max\big \{ \min\{ T_{\ell}:\ell\in [k-1]\}, \min \{S_j: j\in [n+1]\setminus \Lambda_k\}.
   \end{eqnarray*}
   The last equality,  because all the tribes $T_{\ell}\cup\{x_{n+1}\}$ are redundant for $\ell \in [k-1]$.
    \end{proof}

\section{Appendix B: Some algebraic acrobatics}\label{sec:algAcrobatics}

We will perform some algebraic acrobatics to rewrite and equivalent inequality to \eqref{convexity-estimate-entropy1}. Without loss of generality, we will assume that $m=n$, and therefore $[n]\setminus\{m\}=[n-1]$. For $S\in [n-1]$, we have that
$\widehat{f}(S)=\frac12\big (\widehat{f^+}(S)+\widehat{f^-}(S)\big )$, and $\widehat{f}(S\cup\{n\})=\frac12\big (\widehat{f^+}(S)-\widehat{f^-}(S)\big ).$
Hence
we can rewrite the entropy of $f$ as follows
\[ {\bf H}(f) = \sum_{S\subset [n]} \phi \big(\widehat{f}(S)^2\big)= \sum_{S\subset [n-1]} \big [\phi \big (\widehat{f}(S)^2\big )+\phi \big (\widehat{f}(S\cup\{n\})^2\big )\big ].\]
Using Lemma~\ref{lem:binary-entropy}, namely $\phi(x^2)+\phi(y^2)=\phi (x^2+y^2) + (x^2+y^2)\psi \Big ( \frac{x^2}{x^2+y^2}\Big )$, we conclude that
\[{\bf H}(f) = \sum_{S\subset [n-1]} \left [\big (\widehat{f}(S)^2+\widehat{f}(S\cup\{n\})^2\big )\,\psi \bigg (\frac{\widehat{f}(S)^2}{\widehat{f}(S)^2+\widehat{f}(S\cup\{n\})^2}\bigg )+ \phi \big (\widehat{f}(S)^2+\widehat{f}(S\cup\{n\})^2\big )\right ],\]
\begin{eqnarray*}
\mbox{and} \quad \frac12 \big [{\bf H}(f^+) + {\bf H}(f^-)\big ] & = & \frac12 \sum_{S\subset [n-1]} \big (\phi (\widehat{f^+}(S)^2) + \phi (\widehat{f^-}(S)^2)\big ) \\
 &  &\hskip -1.5in =  \, \frac12 \sum_{S\subset [n-1]} \left [ \big (\widehat{f^+}(S)^2 +\widehat{f^-}(S)^2\big )\,\psi \bigg (\frac{\widehat{f^+}(S)^2}{\widehat{f^+}(S)^2 +\widehat{f^-}(S)^2}\bigg )+ \phi \big (\widehat{f^+}(S)^2 +\widehat{f^-}(S)^2\big )\right ],
   \end{eqnarray*}
    Let $\Delta {\bf H}(f)=  {\bf H}(f) -\frac12\big [{\bf H}(f^+)+{\bf H}(f^-)\big ]$.
   Noting that $\widehat{f}(S)^2+\widehat{f}(S\cup\{n\})^2=\frac12(\widehat{f^+}(S)^2+\widehat{f^-}(S)^2)$  (compare to \cite[Proposition 2.3]{WWW14}), and that $\widehat{f}(S)=\frac12(\widehat{f^+}(S)+\widehat{f^-}(S))$, we conclude that
   \begin{eqnarray*}
 \Delta{\bf H}(f)
   &  = & \hskip -.1in \sum_{S\subset [n-1]} \left (  \frac{\widehat{f^+}(S)^2 +\widehat{f^-}(S)^2}{2}\,\left [\psi \bigg (\frac{ \frac12 \big (\widehat{f^+}(S)+\widehat{f^-}(S)\big )^2}{\widehat{f^+}(S)^2 +\widehat{f^-}(S)^2}\bigg ) -\psi \bigg (\frac{\widehat{f^+}(S)^2}{\widehat{f^+}(S)^2+\widehat{f^-}(S)^2}\bigg )\right ]\right. \\
   & & \hskip 1in  \left. +\, \phi \bigg ( \frac{ \widehat{f^+}(S)^2+\widehat{f^-}(S)^2}{2}\bigg ) -\frac12 \phi \big (\widehat{f^+}(S)^2 +\widehat{f^-}(S)^2\big )\right ).
   \end{eqnarray*}
Notice that by \eqref{eq:gx/2=gx/2},
  $\psi (1-x)=\psi (x)$ and $\frac12(c+d)^2+\frac12(c-d)^2=c^2+d^2$, we have that
 \begin{equation}\label{eq:convexityH}
\Delta{\bf H}(f)  = \hskip -.1in
   \sum_{S\subset [n-1]}  \frac{\widehat{f^+}(S)^2 +\widehat{f^-}(S)^2}{2}\left [\psi \bigg (\frac{ \frac12\big (\widehat{f^+}(S)-\widehat{f^-}(S)\big )^2}{\widehat{f^+}(S)^2 +\widehat{f^-}(S)^2}\bigg ) +1 - \psi \bigg (\frac{\widehat{f^+}(S)^2}{\widehat{f^+}(S)^2+\widehat{f^-}(S)^2}\bigg )\right ].
\end{equation}

We will need the following  elementary lemma, proven at the end.

\begin{lemma}\label{lem:elementary}
For all $0\leq t\leq 1$ we have that $0\leq 1-\psi(t)\leq 4\Big (\frac12 -t\big )^2$.
\end{lemma}

We now find an equivalent formulation of   \eqref{convexity-estimate-entropy1}. 
 Lemma~\ref{lem:elementary}   implies that for all $S\subset [n-1]$
\[ 0\leq \frac12 \big (\widehat{f^+}(S)^2+\widehat{f^-}(S)^2\big )\bigg (1-\psi \bigg (\frac{\widehat{f^+}(S)^2}{\widehat{f^+}(S)^2+\widehat{f^-}(S)^2}\bigg )\bigg)\leq \big (\widehat{f^+}(S)-\widehat{f^-}(S)\big )^2,\]
and  when adding over $S\in [n-1]$ it always satisfies the desired estimate
\[ \sum_{S\subset [n-1]}\frac12\big (\widehat{f^+}(S)^2+\widehat{f^-}(S)^2\big )\bigg (1-\psi \bigg (\frac{\widehat{f^+}(S)^2}{\widehat{f^+}(S)^2+\widehat{f^-}(S)^2}\bigg )\bigg)\leq \sum_{S\subset [n-1]}\big (\widehat{f^+}(S)-\widehat{f^-}(S)\big )^2.\]
Therefore, in light of \eqref{eq:convexityH}, estimate~\eqref{convexity-estimate-entropy1} holds if an only if
\begin{equation*} 
 \sum_{S\subset [n-1]}\frac12(\widehat{f^+}(S)^2+\widehat{f^-}(S)^2)\,\psi \bigg (\frac{\frac12\big (\widehat{f^+}(S)-\widehat{f^-}(S)\big )^2}{\widehat{f^+}(S)^2+\widehat{f^-}(S)^2}\bigg )\leq K' \sum_{S\subset [n-1]}\big (\widehat{f^+}(S)-\widehat{f^-}(S)\big )^2,
\end{equation*}
where $K=K'+1$. 
This is~\eqref{final-inequality}, as claimed.

\begin{proof}[Proof of Lemma~\ref{lem:elementary}]
Let $t=\frac12(1+\epsilon)$ for $|\epsilon|\leq 1$, so that $4\big(\frac12-t\big )^2=\epsilon^2$, and
\begin{eqnarray*}
 1-\psi(t)  & = & 1-\frac12(1+\epsilon) \log \frac{2}{1+\epsilon}-\frac12(1-\epsilon) \log \frac{2}{1-\epsilon} \\
               & = & \frac{-1}{2\ln 2}\Big [(1+\epsilon) \ln \frac{1}{1+\epsilon}+(1-\epsilon) \ln \frac{1}{1-\epsilon}\Big ].
\end{eqnarray*}               
  Now, $\ln\frac{1}{1-t}=\sum_{n=1}^{\infty} \frac{t^n}{n}$ for all $|t|<1$, therefore
 \begin{eqnarray*}
 1-\psi(t)  & = &  \frac{-1}{2\ln 2}\Big [(1+\epsilon)\sum_{n=1}^{\infty} \frac{(-1)^n\epsilon^n}{n} +(1-\epsilon) \sum_{n=1}^{\infty} \frac{\epsilon^n}{n}\Big ]\\
               & = &  \frac{1}{\ln 2}\Big [\sum_{k=1}^{\infty} \epsilon^{2k}\Big [\frac{1}{2k-1}-\frac{1}{2k}\Big ]  \Big ]
                \, = \, \frac{1}{\ln 2}\sum_{k=1}^{\infty} \epsilon^{2k} \frac{1}{(2k-1)2k}.
 \end{eqnarray*}              
  Since $|\epsilon |\leq 1$ and $ \sum_{k=1}^{\infty}  \frac{1}{(2k-1)2k} =\ln 2$, this  implies that 
 \[ 0\leq 1-\psi(t) \leq \frac{\epsilon^2}{\ln 2} \sum_{k=1}^{\infty}  \frac{1}{(2k-1)2k}  = \epsilon^2 = 4\Big (\frac12 -t \Big )^2.\]
    
\end{proof}


%


\begin{thebibliography}{09}

\bibitem[ACK+20]{ACK+20} S. Arunachalam, S. Chakraborty, M. Kouch\'y, N. Saurabh, and R. de Wolf, \emph{Improved Bounds on Fourier Entropy and Min-Entropy}. In 37th International Symposium on Theoretical Aspects of Computer Science (STACS 2020). Leibniz International Proceedings in Informatics (LIPIcs), Volume {\bf 154}, pp. 45:1--45:19, Schloss Dagstuhl – Leibniz-Zentrum f\"ur Informatik (2020). 

\bibitem[BL90]{BL90} M. Ben-Or and N. Linial, \emph{Collective coin flipping.} In Silvio Micali, editor, Randomness and Computation. Academic Press, New York, 1990.

\bibitem[CKL+13]{CKL+13} S. Chakraborty, R. Kulkarni, S. Lokam, and N. Saurabh, \emph{Upper bounds on Fourier entropy.} In Electronic Colloquium on Computational Complexity TR13-052, 2013.

\bibitem[FK96]{FK96} E. Friedgut and G. Kalai, \emph{Every monotone graph property has a sharp threshold.} Proc. AMS, {\bf 124}(10):2993--3002, 1996.

\bibitem[GMcP25]{GMcP25} M. J. Gonz\'alez, Paul MacManus, M. C. Pereyra, \emph{Las funciones booleanas y el lema de Bonami}. La Gaceta de la RME {\bf 28}, N\'um. 1, 51--88, 2025.

\bibitem[Han25]{Han25} X. Han, \emph{A New Bound for the Fourier-Entropy-Influence Conjecture}. Combinatorica (2025) 45:4.

\bibitem[KKL88]{KKL88} J. Kahn, G. Kalai, and N. Linial, \emph{The influence of variables on Boolean functions.} In Proceedings of the 29th Annual IEEE Symposium on Foundations of Computer Science, pages 68--80, 1988.

\bibitem[KKL+20]{KKL+20} E. Kelman, G. Kindler, N. Lifshitz, D. Minzer, and M. Safra, \emph{Towards a Proof of the Fourier-Entropy Conjecture?}. Geom. Funct. Anal. 30, 1097--1138 (2020). 

\bibitem[KLW10]{KLW10} A. Klivans, H. Lee, and A. Wan, \emph{Mansour’s Conjecture is true for random DNF formulas.} In Proceedings of the 23rd Annual Conference on Learning Theory, 2010.


\bibitem[O'Do21]{O'Do21} R. O'Donnell, \emph{Analysis of Boolean Functions}. Originally published April 2014 by Cambridge University Press. May 2021 arXiv edition - arXiv: 2105.10386v1 

\bibitem[OT13]{OT13} R. O’Donnell and L.-Y. Tan, \emph{A composition theorem for the Fourier Entropy- Influence conjecture.} In Proceedings of the 40th International Colloquium on Automata, Languages and Programming, pages 780--791, 2013.

\bibitem[OWZ11]{OWZ11} R.  O’Donnell, J. Wright, and  Y. Zhou, \emph{The Fourier Entropy–Influence Conjecture for Certain Classes of Boolean Functions.} In: Aceto, L., Henzinger, M., Sgall, J. (eds) Automata, Languages and Programming. ICALP 2011. Lecture Notes in Computer Science {\bf 6755}. Springer, Berlin, Heidelberg, 2011.

\bibitem[WWW14]{WWW14} A. Wan, J. Wright, and C. Wu, \emph {Decision Trees, Protocols, and
the Fourier Entropy-Influence Conjecture}. ITCS’14, 67--80, 2014.

\end{thebibliography}
\end{document}